%% file: nonconvex_svrg.tex
\documentclass{article}

\usepackage{amsmath,amssymb,enumerate,comment}
\usepackage{amsthm,cancel}
\usepackage[usenames,dvipsnames]{color}
\usepackage{natbib}

\usepackage[margin=1.25in]{geometry}

\usepackage{graphicx} 
\usepackage{subfigure} 
\usepackage{multirow}
\usepackage{xcolor}
\usepackage{mathtools}
\usepackage{relsize}
\usepackage{url}
\usepackage{nicefrac}
\usepackage{xspace}
\usepackage{algorithm}
\usepackage{algorithmic}

\newtheorem{lemma}{Lemma}
\newtheorem{theorem}{Theorem}
\newtheorem{corollary}{Corollary}

\newtheorem{definition}{Definition}

\newtheorem*{lemma*}{Lemma}
\newtheorem*{theorem*}{Theorem}
\newtheorem*{corollary*}{Corollary}
\newtheorem*{remark*}{Remark}
\newtheorem*{definition*}{Definition}

\newcommand{\sgd}{\textsc{Sgd}\xspace}
\newcommand{\sdca}{\textsc{Sdca}}

\newcommand{\svrg}{\textsc{Svrg}\xspace}
\newcommand{\msvrg}{\textsc{Msvrg}}
\newcommand{\gd}{\textsc{GradientDescent}\xspace}
\newcommand{\sml}[1]{{\small #1}}
\newcommand{\fromto}[3]{\sml{$#1{\;\le\;}#2{\;\le\;}#3$}}

\newcommand{\reals}{\mathbb{R}}

\newcommand{\nlsum}{\sum\nolimits}
\newcommand{\E}{\mathbb{E}}
\newcommand{\Fc}{\mathcal{F}}

\renewcommand{\cite}[1]{\citep{#1}}
\setcitestyle{authoryear,round,citesep={;},aysep={,},yysep={;}}

\begin{document}

\title{Stochastic Variance Reduction for Nonconvex Optimization}

\author{\\
Sashank J. Reddi\\
\texttt{sjakkamr@cs.cmu.edu}\\
Carnegie Mellon University \\
\and \\
Ahmed Hefny\\
\texttt{ahefny@cs.cmu.edu}\\
Carnegie Mellon University \\
\and \\
Suvrit Sra\\
\texttt{suvrit@mit.edu} \\
Massachusetts Institute of Technology \\
\and \\
Barnab\'{a}s P\'{o}cz\'os\\
\texttt{bapoczos@cs.cmu.edu} \\
Carnegie Mellon University \\
\and \\
Alex Smola \\
\texttt{alex@smola.org} \\ 
Carnegie Mellon University \\
}

%

%
%
%
%
%
\date{Original circulated date: 5th February, 2016.}
\maketitle

\begin{abstract}
  We study nonconvex finite-sum problems and analyze stochastic variance reduced gradient (\svrg) methods for them. \svrg and related methods have recently surged into prominence for convex optimization given their edge over stochastic gradient descent (\sgd); but their theoretical analysis almost exclusively assumes convexity. In contrast, we prove non-asymptotic rates of convergence (to stationary points) of \svrg for nonconvex optimization, and show that it is provably faster than \sgd and gradient descent. We also analyze a subclass of nonconvex problems on which \svrg attains linear convergence to the \emph{global} optimum. We extend our analysis to mini-batch variants of \svrg, showing (theoretical) linear speedup due to mini-batching in parallel settings. 
\end{abstract} 

\section{Introduction}
\label{sec:intro}
We study nonconvex \emph{finite-sum} problems of the form
\begin{equation}
  \label{eq:1}
  \min_{x\in \reals^d}\ f(x) := \frac{1}{n}\sum_{i=1}^n f_i(x),
\end{equation}
where neither $f$ nor the individual $f_i$ ($i \in [n]$) are necessarily convex; just Lipschitz smooth (i.e., Lipschitz continuous gradients). We use $\mathcal{F}_n$ to denote all functions of the form~\eqref{eq:1}. We optimize such functions in the Incremental First-order Oracle (IFO) framework~\citep{agarwal2014} defined below.
\begin{definition}
  For $f \in \mathcal{F}_n$, an IFO takes an index $i \in [n]$ and a point $x \in \mathbb{R}^d$, and returns the pair $(f_i(x),\nabla f_i(x))$. 
\end{definition}
IFO based complexity analysis was introduced to study lower bounds for finite-sum problems. Algorithms that use IFOs are favored in large-scale applications as they require only a small amount first-order information at each iteration. Two fundamental models in machine learning that profit from IFO algorithms are (i) empirical risk minimization, which typically uses convex finite-sum models; and (ii) deep learning, which uses nonconvex ones. 

The prototypical IFO algorithm, stochastic gradient descent (\sgd)\footnote{We use  `incremental gradient' and `stochastic gradient' interchangeably, though we are only interested in finite-sum problems.} has witnessed tremendous progress in the recent years. By now a variety of accelerated, parallel, and faster converging versions are known. Among these, of particular importance are variance reduced (VR) stochastic methods~\citep{Schmidt13,Johnson13,Defazio14}, which have delivered exciting progress such as linear convergence rates (for strongly convex functions) as opposed to sublinear rates of ordinary \sgd~\citep{RobMon51,nemirov09}. Similar (but not same) benefits of VR methods can also be seen in smooth convex functions. The \svrg algorithm of~\citep{Johnson13} is particularly attractive here because of its low storage requirement in comparison to the algorithms in~\citep{Schmidt13,Defazio14}. 

Despite the meteoric rise of VR methods, their analysis for general nonconvex problems is largely missing. \citet{Johnson13} remark on convergence of \svrg when $f \in \mathcal{F}_n$ is locally strongly convex and provide compelling experimental results (Fig.~4 in~\citep{Johnson13}). However, problems encountered in practice are typically not even locally convex, let alone strongly convex. The current analysis of \svrg does not extend to nonconvex functions as it relies heavily on convexity for controlling the variance. Given the dominance of stochastic gradient methods in optimizing deep neural nets and other large nonconvex models, theoretical investigation of faster nonconvex stochastic methods is much needed.

\begin{table*}[t]\small
\label{tab:algo-comparison}
\begin{center}
\resizebox{\textwidth}{!}{
\bgroup
\def\arraystretch{1.5}
\begin{tabular}{|l|c|c|c|c|}
\hline
Algorithm & Nonconvex & Convex & Gradient Dominated & Fixed Step Size?\\
\hline
$\sgd$    & $O\left(1/\epsilon^2\right)$ & $O\left(1/\epsilon^2\right)$ & $O\left(1/\epsilon^2\right)$ & $\times$ \\
$\gd$    & $O\left(n/\epsilon\right)$ & $O\left(n/\epsilon\right)$ & $O\left(n \tau \log(1/\epsilon)\right)$& $\surd$ \\
$\svrg$ & $\color{red} O\bigl(n + (n^{2/3}/\epsilon)\bigr)$ & $\color{red} O\bigl(n + (\sqrt{n}/\epsilon)\bigr)$ & $\color{red} O\bigl((n + n^{2/3} \tau)\log(1/\epsilon)\bigr)$ & $\surd$\\
$\msvrg$ & $\color{red} O\bigl(\min\big\lbrace 1/\epsilon^2, n^{2/3}/\epsilon\bigr\rbrace\bigr)$ & $\color{red} O\bigl(\min\bigl\lbrace 1/\epsilon^2, \sqrt{n}/\epsilon\bigr\rbrace\bigr)$ & $-$ & $\times$\\
\hline
\end{tabular}
\egroup
}
\end{center}
\caption{\small Table comparing the IFO complexity of different algorithms discussed in the paper. The complexity is measured in terms of the number of oracle calls required to achieve an $\epsilon$-accurate solution (see Definition~\ref{def:eps-accurate}). Here, by fixed step size, we mean that the step size of the algorithm is fixed and does not dependent on $\epsilon$ (or alternatively $T$, the total number of iterations). The complexity of gradient dominated functions refers to the number of IFO calls required to obtain $\epsilon$-accurate solution for a $\tau$-gradient dominated function (see Section~\ref{sec:background} for the definition). For $\sgd$, we are not aware of any specific results for gradient dominated functions. Also, $[f(x^0) - f(x^*)]$ and $\|x^0 - x^*\|$ (where $x^0$ is the initial point and $x^*$ is an optimal solution to~\eqref{eq:1}) are assumed to be constant for a clean comparison. The results marked in red are the contributions of this paper.}
\end{table*}

Convex VR methods are known to enjoy the faster convergence rate of \gd but with a much weaker dependence on $n$,  without compromising the rate like \sgd. However, it is not clear if these benefits carry beyond convex problems, prompting the central question of this paper:
\begin{quote}
  \it 
  For nonconvex functions in $\Fc_n$, can one achieve convergence rates faster than both \sgd and \gd using an IFO? If so, then how does the rate depend on $n$ and on the number of iterations performed by the algorithm? 
\end{quote}
Perhaps surprisingly, we provide an affirmative answer to this question by showing that a careful selection of parameters in $\svrg$ leads to faster convergence than both $\sgd$ and $\gd$. To our knowledge, ours is the \emph{first} work to improve convergence rates of \sgd and \gd for IFO-based nonconvex optimization. 

\textbf{Main Contributions.} We summarize our main contributions below and also list the key results in Table~\ref{tab:algo-comparison}.
\begin{list}{$\bullet$}{\leftmargin=1em}
\setlength{\itemsep}{-2pt}
\item We analyze nonconvex stochastic variance reduced gradient (\svrg), and prove  that it has faster rates of convergence than $\gd$ and ordinary \sgd. We show that $\svrg$ is faster than $\gd$ by a factor of $n^{1/3}$ (see Table~\ref{tab:algo-comparison}).
\item We provide new theoretical insights into the interplay between step-size, iteration complexity and convergence of nonconvex \svrg (see Corollary~\ref{cor:svrg-nonconvex-oracle-gen}).
\item For an interesting nonconvex subclass of $\Fc_n$ called gradient dominated functions \cite{Polyak1963,nesterov2006}, we propose a variant of \svrg that attains a \emph{global} linear rate of convergence. We improve upon many prior results for this subclass of functions (see Section~\ref{sec:grad-dom}). To the best of our knowledge, ours is the first work that shows a stochastic method with linear convergence for gradient dominated functions.
\item We analyze mini-batch nonconvex $\svrg$ and show that it provably benefits from mini-batching. Specifically, we show theoretical \emph{linear} speedups in parallel settings for large mini-batch sizes. By using a mini-batch of size $b$ $(< n^{2/3})$, we show that mini-batch nonconvex $\svrg$ is faster by a factor of $b$ (Theorem~\ref{thm:nonconvex-minibatch}). We are not aware of any prior work on mini-batch first-order stochastic methods that shows linear speedup in parallel settings for nonconvex optimization.
\item Our analysis yields as a byproduct a direct convergence analysis for $\svrg$ for  smooth convex functions (Section~\ref{sec:convex}).
\item We examine a variant of $\svrg$ (called $\msvrg$) that has faster rates than both $\gd$ and $\sgd$.
\end{list}

\subsection{Related Work}
\label{sec:relatedwork}

\textbf{Convex.} \citet{bertsekas11.survey} surveys several incremental gradient methods for convex problems. A key reference for stochastic convex optimization (for $\min \E_z[F(x,z)]$) is~\citep{nemirov09}. Faster rates of convergence are attained for problems in $\mathcal{F}_n$ by VR methods, see e.g.,~\citep{Defazio14,Johnson13,Schmidt13,Konecny15,sdca,defazio2014finito}.  Asynchronous VR frameworks are developed in~\citep{Reddi2015}. \citet{agarwal2014,lan2015} study lower-bounds for convex finite-sum problems.
\citet{Shwartz15} prove linear convergence of stochastic dual coordinate ascent when the individual $f_i$ ($i\in[n]$) are nonconvex but $f$ is strongly convex. They do not study the general nonconvex case. Moreover, even in their special setting our results improve upon theirs for the high condition number regime.

\textbf{Nonconvex.} \sgd dates at least to the seminal work~\citep{RobMon51}; and since then it has been developed in several directions~\citep{poljak1973,ljung1977,bot91,kushner2012}. In the (nonsmooth) finite-sum setting, \citet{Sra2012} considers proximal splitting methods, and analyzes asymptotic convergence with nonvanishing gradient errors. \citet{hong2014} studies a distributed nonconvex incremental ADMM algorithm. 

These works, however, only prove expected convergence to stationary points and often lack analysis of rates. The first nonasymptotic convergence rate analysis for \sgd is in~\citep{Ghadimi13}, who show that \sgd ensures $\|\nabla f\|^2 \le \epsilon$ in $O(1/\epsilon^2)$ iterations. A similar rate for parallel and distributed \sgd was shown recently in~\citep{Lian2015}. \gd is known to ensure $\|\nabla f\|^2\le \epsilon$ in $O(1/\epsilon)$ iterations~\citep[Chap.~1.2.3]{nesterov03}.

The first analysis of nonconvex \svrg seems to be due to~\citet{shamir2014stochastic}, who considers the special problem of computing a few leading eigenvectors (e.g., for PCA); see also the follow up work~\citep{shamir2015fast}. Finally, we note another interesting example, stochastic optimization of locally quasi-convex functions~\citep{hazan2015}, wherein actually a $O(1/\epsilon^2)$ convergence in function value is shown.  
\section{Background \& Problem Setup}
\label{sec:background}
We say $f$ is $L$-\emph{smooth} if there is a constant $L$ such that
\begin{equation*}
  \|\nabla f(x)-\nabla f(y)\| \le L\|x-y\|,\quad\forall\ x, y \in \reals^d.
\end{equation*}
Throughout, we assume that the functions $f_i$ in~\eqref{eq:1} are $L$-smooth, so that $\|\nabla f_i(x)-\nabla f_i(y)\| \le L\|x-y\|$ for all $i \in [n]$. Such an assumption is very common in the analysis of first-order methods. Here the Lipschitz constant $L$ is assumed to be independent of $n$. A function $f$ is called $\lambda$-\emph{strongly convex} if there is $\lambda\ge 0$ such that
$$
f(x) \geq f(y) + \langle \nabla f(y), x - y \rangle + \tfrac{\lambda}{2} \|x - y\|^2\quad\forall x, y \in \mathbb{R}^d.
$$
The quantity $\kappa := L/\lambda$ is called the \emph{condition number} of $f$, whenever $f$ is $L$-smooth and $\lambda$-strongly convex. We say $f$ is non-strongly convex when $f$ is $0$-strongly convex. 

We also recall the class of gradient dominated functions~\citep{Polyak1963,nesterov2006}, where a function $f$ is called $\tau$-\emph{gradient dominated} if for any $x\in \reals^d$ 
\begin{equation}
  \label{eq:2}
  f(x) - f(x^*) \leq \tau \|\nabla f(x)\|^2,
\end{equation}
where $x^*$ is a global minimizer of $f$. Note that such a function $f$ need not be convex; it is also easy to show that a $\lambda$-strongly convex function is $1/2\lambda$-gradient dominated.

We analyze convergence rates for the above classes of functions. Following \citet{nesterov03,Ghadimi13} we use $\|\nabla f(x)\|^2 \le \epsilon$ to judge when is iterate $x$ approximately stationary. Contrast this with \sgd for convex $f$, where one uses $[f(x) - f(x^*)]$ or $\|x - x^*\|^2$ as a convergence criterion. Unfortunately, such criteria cannot be used for nonconvex functions due to the hardness of the problem. While the quantities $\|\nabla f(x)\|^2$ and $f(x) - f(x^*)$ or $\|x - x^*\|^2$ are not comparable in general (see \citep{Ghadimi13}), they are typically assumed to be of similar magnitude. Throughout our analysis, we do \emph{not} assume $n$ to be constant, and report dependence on it in our results. For our analysis, we need the following definition.
\begin{definition}
  \label{def:eps-accurate}
  A point $x$ is called $\epsilon$-accurate if $\|\nabla f(x)\|^2 \leq \epsilon$. A stochastic iterative algorithm is said to achieve $\epsilon$-accuracy in $t$ iterations if $\mathbb{E}[\|\nabla f(x^t)\|^2] \leq \epsilon$, where the expectation is over the stochasticity of the algorithm.
\end{definition}
We introduce one more definition useful in the analysis of \sgd methods for bounding the variance. 
\begin{definition}
  We say $f \in \mathcal{F}_n$ has a $\sigma$-bounded gradient if $\|\nabla f_i(x)\| \leq \sigma$ for all $i \in [n]$ and $x \in \mathbb{R}^d$.
\end{definition}

\subsection{Nonconvex SGD: Convergence Rate}
\label{sec:sgd}
Stochastic gradient descent (\sgd) is one of the simplest algorithms for solving \eqref{eq:1}; Algorithm~\ref{alg:sgd} lists its pseudocode.
\begin{algorithm}[h]\small
   \caption{SGD}
   \label{alg:sgd}
\begin{algorithmic}
   \STATE {\bfseries Input:} $x^0 \in \mathbb{R}^d$, Step-size sequence: $\{\eta_t > 0\}_{t=0}^{T-1}$
   \FOR{$t=0$ {\bfseries to} $T-1$}
   \STATE Uniformly randomly pick $i_t$ from $\{1, \dots, n\}$
   \STATE $x^{t+1} = x^{t} - \eta_t \nabla f_{i_t}(x)$
   \ENDFOR
\end{algorithmic}
\end{algorithm}
By using a uniformly randomly chosen (with replacement) index $i_t$  from $[n]$, $\sgd$ uses an unbiased estimate of the gradient at each iteration. Under appropriate conditions, \citet{Ghadimi13} establish convergence rate of \sgd to a stationary point of $f$. Their results include the following theorem.

\begin{theorem}
  \label{thm:sgd-conv}
  Suppose $f$ has $\sigma$-bounded gradient; let $\eta_t = \eta = c/\sqrt{T}$ where $c = \sqrt{\tfrac{2(f(x^0) - f(x^*))}{L\sigma^2}}$, 
  and $x^*$ is an optimal solution to~\eqref{eq:1}. Then, the iterates of Algorithm~\ref{alg:sgd} satisfy
  \begin{align*}
    \min_{0 \leq t \leq T-1} \mathbb{E}[\|\nabla f(x^t)\|^2] \leq
    \sqrt{\frac{2(f(x^0) - f(x^*)) L}{T}}\sigma.
  \end{align*}
\end{theorem}
For completeness we present a proof in the appendix. Note that our choice of step size $\eta$ requires knowing the total number of iterations $T$ in advance. A more practical approach is to use a $\eta_t \propto 1/\sqrt{t}$ or $1/t$. A bound on IFO calls made by Algorithm~\ref{alg:sgd} follows as a corollary of Theorem~\ref{thm:sgd-conv}.
\begin{corollary}
  \label{cor:sgd-oracle}
  Suppose function $f$ has $\sigma$-bounded gradient, then the IFO complexity of Algorithm~\ref{alg:sgd} to obtain an $\epsilon$-accurate solution is $O(1/\epsilon^2)$.
\end{corollary}
As seen in Theorem~\ref{thm:sgd-conv}, $\sgd$ has a convergence rate of $O(\nicefrac{1}{\sqrt{T}})$. This rate is not improvable in general even when the function is (non-strongly) convex~\citep{nemYud83}. This barrier is due to the variance introduced by the stochasticity of the gradients, and it is not clear if better rates can be obtained \sgd even for convex $f\in\mathcal{F}_n$. 
 
\section{Nonconvex SVRG}
\label{sec:nsvrg}
We now turn our focus to variance reduced methods. We use \svrg~\citep{Johnson13}, an algorithm recently shown to be very effective for reducing variance in convex problems. As a result, it has gained considerable interest in both machine learning and optimization communities. We seek to understand its benefits for \emph{nonconvex} optimization. For reference, Algorithm~\ref{alg:svrg} presents \svrg's pseudocode.

Observe that Algorithm~\ref{alg:svrg} operates in epochs. At the end of epoch $s$, a full gradient is calculated at the point $\tilde{x}^{s}$, requiring $n$ calls to the IFO. Within its inner loop \svrg performs $m$ stochastic updates. The total number of IFO calls for each epoch is thus $\Theta(m +  n)$. For $m = 1$, the algorithm reduces to the classic $\gd$ algorithm. Suppose $m$ is chosen to be $O(n)$ (typically used in practice), then the total IFO calls per epoch is $\Theta(n)$. To enable a fair comparison with $\sgd$, we assume that the total number of inner iterations across all epochs in Algorithm~\ref{alg:svrg} is $T$. Also note a simple but important implementation detail: as written,  Algorithm~\ref{alg:svrg} requires storing all the iterates $x^{s+1}_t$ (\fromto{0}{t}{m}). This storage can be avoided by keeping a running average with respect to the probability distribution $\{p_i\}_{i=0}^m$.

Algorithm~\ref{alg:svrg} attains linear convergence for strongly convex $f$~\citep{Johnson13}; for non-strongly convex functions, rates faster than \sgd can be shown by using an indirect perturbation argument---see e.g.,~\citep{Konecny2013,Xiao14}. 

\begin{algorithm}[tb]\small
   \caption{SVRG$\left(x^0,T, m, \{p_i\}_{i=0}^{m}, \{\eta_i\}_{i=0}^{m-1}\right)$}
   \label{alg:svrg}
\begin{algorithmic}[1]
   \STATE {\bfseries Input:} $\tilde{x}^0 = x^0_m = x^0 \in \mathbb{R}^d$,  epoch length $m$, step sizes $\{\eta_i > 0\}_{i=0}^{m-1}$, $S = \lceil T/m \rceil$, discrete probability distribution $\{p_i\}_{i=0}^{m}$
   \FOR{$s=0$ {\bfseries to} $S-1$}
   \STATE $x^{s+1}_0 = x^{s}_m$
   \STATE $g^{s+1} = \frac{1}{n} \sum_{i=1}^n \nabla f_{i}(\tilde{x}^{s})$
   \FOR{$t=0$ {\bfseries to} $m-1$}
   \STATE Uniformly randomly pick $i_t$ from $\{1, \dots, n\}$ \label{alg1:ln:sample}
   \STATE $v_t^{s+1} =  \nabla f_{i_t}(x^{s+1}_t) - \nabla f_{i_t}(\tilde{x}^{s}) + g^{s+1}$ \label{alg1:ln:update}
   \STATE $x^{s+1}_{t+1} = x^{s+1}_{t} - \eta_t v_t^{s+1} $
   \ENDFOR
   \STATE $\tilde{x}^{s+1} = \sum_{i=0}^{m} p_i x_{i}^{s+1}$
   \ENDFOR
   \STATE {\bfseries Output:} Iterate $x_a$ chosen uniformly random from $\{\{x^{s+1}_t\}_{t=0}^{m-1}\}_{s=0}^{S-1}$.
\end{algorithmic}
\end{algorithm}

We first state an intermediate result for the iterates of nonconvex \svrg. To ease exposition, we define
\begin{align}
\label{eq:Gamma-t}
  \Gamma_t = \bigl(\eta_t - \frac{c_{t+1}\eta_t}{\beta_t} - \eta_t^2L - 2c_{t+1}\eta_t^2\bigr),
\end{align}
for some parameters $c_{t+1}$ and $\beta_t$ (to be defined shortly).
Our first main result is the following theorem that provides convergence rate of Algorithm~\ref{alg:svrg}.
\begin{theorem}
  Let $f \in \Fc_n$. Let $c_m = 0$, $\eta_t = \eta > 0$, $\beta_t = \beta > 0$, and $c_{t} = c_{t+1}(1 + \eta\beta + 2\eta^2L^2 ) +  \eta^2L^3$ such that $\Gamma_t > 0$ for \fromto{0}{t}{m-1}. Define the quantity $\gamma_n := \min_t \Gamma_t$. 
  Further, let $p_{i} = 0$ for $0{\;\leq\;}i{\;<\;}m$ and $p_{m} = 1$, and let $T$ be a multiple of $m$. Then for the output $x_a$ of Algorithm~\ref{alg:svrg} we have
  \begin{align*}
    \mathbb{E}[\|\nabla f(x_a)\|^2] \leq \frac{f(x^{0}) - f(x^*)}{T\gamma_n},
  \end{align*}
  where $x^*$ is an optimal solution to~\eqref{eq:1}.
  \label{thm:nonconvex-inter}
\end{theorem}

Furthermore, we can also show that nonconvex \svrg exhibits expected descent (in objective) after every epoch. The condition that $T$ is a multiple of $m$ is solely for convenience and can be removed by slight modification of the theorem statement. Note that the value $\gamma_n$ above can depend on $n$. To obtain an explicit dependence, we simplify it using specific choices for $\eta$ and $\beta$, as formalized below. 

\begin{theorem}
  Suppose $f \in \Fc_n$. Let $\eta = \mu_0/(Ln^{\alpha})$ ($0 < \mu_0 < 1$ and $0 < \alpha \leq 1$), $\beta = L/n^{\alpha/2}$, $m = \lfloor n^{3\alpha/2}/(3\mu_0) \rfloor$ and $T$ is some multiple of $m$. Then there exists universal constants $\mu_0, \nu > 0$ such that we have the following: $\gamma_n \geq \frac{\nu}{Ln^{\alpha}}$ in Theorem~\ref{thm:nonconvex-inter} and
  \begin{align*}
    \mathbb{E}[\|\nabla f(x_a)\|^2] &\leq \frac{Ln^{\alpha} [f(x^{0}) - f(x^*)]}{T\nu},
  \end{align*} 
  where $x^*$ is an optimal solution to the problem in~\eqref{eq:1} and $x_a$ is the output of Algorithm~\ref{alg:svrg}.
  \label{thm:nonconvex-gen}
\end{theorem}

By rewriting the above result in terms IFO calls, we get the following general corollary for nonconvex \svrg.

\begin{corollary}
  \label{cor:svrg-nonconvex-oracle-gen}
  Suppose $f \in \Fc_n$. Then the IFO complexity of Algorithm~\ref{alg:svrg} (with parameters from Theorem~\ref{thm:nonconvex-gen}) for achieving an $\epsilon$-accurate solution is:
  \begin{align*}
    \text{IFO calls} =
    \begin{cases}
      O\left(n + (n^{1 - \frac{\alpha}{2}}/\epsilon)\right),  & \text{if } \alpha < 2/3, \\
      O\left(n + (n^{\alpha}/\epsilon)\right), & \text{if } \alpha \geq 2/3.
    \end{cases}
  \end{align*}
\end{corollary}
%

Corollary~\ref{cor:svrg-nonconvex-oracle-gen} shows the interplay between step size and the IFO complexity. We observe that the number of IFO calls is minimized in Corollary~\ref{cor:svrg-nonconvex-oracle-gen} when $\alpha = 2/3$. This gives rise to the following key results of the paper.

\begin{corollary}
  Suppose $f \in \Fc_n$. Let $\eta = \mu_1/(Ln^{2/3})$ ($0 < \mu_1 < 1$), $\beta = L/n^{1/3}$, $m = \lfloor n/(3\mu_1) \rfloor$ and $T$ is some multiple of $m$. Then there exists universal constants $\mu_1, \nu_1 > 0$ such that we have the following: $\gamma_n \geq \frac{\nu_1}{Ln^{2/3}}$ in Theorem~\ref{thm:nonconvex-inter} and
\begin{align*}
\mathbb{E}[\|\nabla f(x_a)\|^2] &\leq \frac{Ln^{2/3} [f(x^{0}) - f(x^*)]}{T\nu_1},
\end{align*} 
where $x^*$ is an optimal solution to the problem in~\eqref{eq:1} and $x_a$ is the output of Algorithm~\ref{alg:svrg}.
\label{cor:nonconvex}
\end{corollary}

\begin{corollary}
\label{cor:svrg-nonconvex-oracle}
If $f \in \Fc_n$, then the IFO complexity of Algorithm~\ref{alg:svrg} (with parameters in Corollary~\ref{cor:nonconvex}) to obtain an $\epsilon$-accurate solution is $O(n + (n^{2/3}/\epsilon))$.
\end{corollary}

Note the rate of $O(1/T)$ in the above results, as opposed to slower $O(1/\sqrt{T})$ rate of $\sgd$ (Theorem~\ref{thm:sgd-conv}). For a more comprehensive comparison of the rates, refer to Section~\ref{sec:comparison}.

\subsection{Gradient Dominated Functions}
\label{sec:grad-dom}
\begin{algorithm}[tb]\small
   \caption{GD-SVRG$\left(x^0, K, T, m, \{p_i\}_{i=0}^{m}, \{\eta_i\}_{i=0}^{m-1}\right)$}
   \label{alg:gd-svrg}
\begin{algorithmic}
   \STATE {\bfseries Input:} $x^0 \in \mathbb{R}^d$, $K$,  epoch length $m$, step sizes $\{\eta_i > 0\}_{i=0}^{m-1}$, discrete probability distribution $\{p_i\}_{i=0}^{m}$
   \FOR{$k=0$ to $K$}
   \STATE $x^{k} = \text{SVRG}(x^{k-1},T, m, \{p_i\}_{i=0}^{m}, \{\eta_i\}_{i=0}^{m-1})$
   \ENDFOR
   \STATE {\bfseries Output:} $x^K$
\end{algorithmic}
\end{algorithm}

Before ending our discussion on convergence of nonconvex \svrg, we prove a linear convergence rate for the class of $\tau$-gradient dominated functions~\eqref{eq:2}. 
For ease of exposition, assume that $\tau > n^{1/3}$, a property analogous to the ``high condition number regime'' for strongly convex functions typical in machine learning. Note that gradient dominated functions can be nonconvex. 

\begin{theorem}
  \label{thm:gd-svrg-thm1}
  Suppose $f$ is $\tau$-gradient dominated where $\tau > n^{1/3}$. Then, the iterates of Algorithm~\ref{alg:gd-svrg} with $T = \lceil 2L\tau n^{2/3}/\nu_1 \rceil$, $m = \lfloor n/(3\mu_1) \rfloor$, $\eta_t = \mu_1/(Ln^{2/3})$ for all $0 \leq t \leq m-1$ and $p_{m} = 1$ and $p_i = 0$ for all $0 \leq i < m$ satisfy
  $$\mathbb{E}[\|\nabla f(x^k)\|^2] \leq 2^{-k}[\| \nabla f(x^{0}) \|^2].
  $$
  Here $\mu_1$ and $\nu_1$ are the constants used in Corollary~\ref{cor:nonconvex}.
\end{theorem}
In fact, for $\tau$-gradient dominated functions we can prove a stronger result of \emph{global} linear convergence.

\begin{theorem}
  \label{thm:gd-svrg-thm2}
  If $f$ is $\tau$-gradient dominated ($\tau > n^{1/3}$), then with $T = \lceil 2L\tau n^{2/3}/\nu_1 \rceil$, $m = \lfloor n/(3\mu_1) \rfloor$, $\eta_t = \mu_1/(Ln^{2/3})$ for $0 \leq t \leq m-1$ and $p_{m} = 1$ and $p_i = 0$ for all $0 \leq i < m$, the iterates of Algorithm~\ref{alg:gd-svrg} satisfy
  $$\mathbb{E}[f(x^k) - f(x^*)] \leq 2^{-k}[f(x^0) - f(x^*)].$$
  Here $\mu_1$, $\nu_1$ are as in Corollary~\ref{cor:nonconvex}; $x^*$ is an optimal solution.
\end{theorem}
An immediate consequence is the following.
\begin{corollary}
  \label{cor:svrg-gd-oracle}
  If $f$ is $\tau$-gradient dominated, the IFO complexity of Algorithm~\ref{alg:gd-svrg} (with parameters from Theorem~\ref{thm:gd-svrg-thm1}) to compute an $\epsilon$-accurate solution is $O((n + \tau n^{2/3}) \log(1/\epsilon))$.
\end{corollary}

Note that $\gd$ can also achieve linear convergence rate for gradient dominated functions~\citep{Polyak1963}. However, $\gd$ requires $O(n +  n \tau \log(1/\epsilon))$ IFO calls to obtain an $\epsilon$-accurate solution as opposed to $O(n + n^{2/3} \tau \log(1/\epsilon))$ for \svrg. Similar (but not the same) gains can be seen for \svrg for strongly convex functions~\citep{Johnson13}. Also notice that we did not assume anything except smoothness on the \emph{individual} functions $f_i$ in the above results. In particular, the following corollary is also an immediate consequence.

\begin{corollary}
  \label{cor:svrg-gd-strongly-convex}
  If $f$ is $\lambda$-strongly convex and the functions $\{f_i\}_{i=1}^n$ are possibly nonconvex, then the number of IFO calls made by Algorithm~\ref{alg:gd-svrg} (with parameters from Theorem~\ref{thm:gd-svrg-thm1}) to compute an $\epsilon$-accurate solution is $O((n + n^{2/3} \kappa) \log(1/\epsilon))$.
\end{corollary}

Recall that here $\kappa$ denotes the condition number $L/\lambda$ for a $\lambda$-strongly convex function. Corollary~\ref{cor:svrg-gd-strongly-convex} follows from Corollary~\ref{cor:svrg-gd-oracle} upon noting that $\lambda$-strongly convex function is $1/2\lambda$-gradient dominated. Theorem~\ref{thm:gd-svrg-thm2} generalizes the linear convergence result in \citep{Johnson13} since it allows nonconvex $f_i$. Observe that Corollary~\ref{cor:svrg-gd-strongly-convex} also applies when $f_i$ is strongly convex for all $i \in [n]$, though in this case a more refined result can be proved~\citep{Johnson13}.

Finally, we note that our result also improves on a recent result on $\sdca$ in the setting of Corollary~\ref{cor:svrg-gd-strongly-convex} when the condition number $\kappa$ is reasonably large -- a case that typically arises in machine learning. More precisely, for $l_2$-regularized empirical loss minimization, \citet{Shwartz15} show that \sdca\ requires $O((n + \kappa^2)\log(1/\epsilon)$ iterations when the $f_i$'s are possibly nonconvex but their sum $f$ is strongly convex. In comparison, we show that Algorithm~\ref{alg:gd-svrg} requires $O((n + n^{2/3}\kappa) \log(1/\epsilon))$ iterations, which is an improvement over \sdca\ when $\kappa > n^{2/3}$.

\section{Convex Case}
\label{sec:convex}
In the previous section, we showed nonconvex $\svrg$ converges to a stationary point at the rate $O(n^{2/3}/T)$. A natural question is whether this rate can be improved if we assume convexity? We provide an affirmative answer. For non-strongly convex functions, this yields a \emph{direct} analysis (i.e., not based on strongly convex perturbations) for $\svrg$. While we state our results in terms of stationarity gap $\|\nabla f(x)\|^2$ for the ease of comparison, our analysis also provides rates with respect to the optimality gap $[f(x) - f(x^*)]$ (see the proof of Theorem~\ref{thm:convex} in the appendix).  
\begin{theorem}
  \label{thm:convex}
  If $f_i$ is convex for all $i \in [n]$, $p_i = 1/m$ for \fromto{0}{i}{m-1}, and $p_m = 0$, then for Algorithm~\ref{alg:svrg}, we have
  \small
  \begin{align*}
    &\mathbb{E}[\|\nabla f(x_a)\|^2] \leq \frac{L\|x^{0} - x^*\|^2 + 4mL^2\eta^2 [f(x^{0}) - f(x^*)]}{T\eta(1 - 4L\eta)},
  \end{align*}
  \normalsize
  where $x^*$ is optimal for~\eqref{eq:1} and $x_a$ is the output of Algorithm~\ref{alg:svrg}.
\end{theorem}

We now state corollaries of this theorem that explicitly show the dependence on $n$ in the convergence rates.

\begin{corollary}
  \label{cor:convex-non-const}
  If $m = n$ and $\eta = 1/(8L\sqrt{n})$ in Theorem~\ref{thm:convex}, then we have the following bound:
  \small
  \begin{align*}
    &\mathbb{E}[\|\nabla f(x_a)\|^2] \leq \frac{L\sqrt{n}(16L\|x^{0} - x^*\|^2 + [f(x^{0}) - f(x^*)])}{T},
  \end{align*}
  \normalsize
  where $x^*$ is optimal for~\eqref{eq:1} and $x_a$ is the output of Algorithm~\ref{alg:svrg}.
\end{corollary}
The above result uses a step size that depends on $n$. For the convex case, we can also use step sizes independent of $n$. The following corollary states the associated result.

\begin{corollary}
  \label{cor:convex-const}
  If $m = n$ and $\eta = 1/(8L)$ in Theorem~\ref{thm:convex}, then we have the following bound:
  \small
  \begin{align*}
    &\mathbb{E}[\|\nabla f(x_a)\|^2] \leq \frac{L(16L\|x^{0} - x^*\|^2 + n [f(x^{0}) - f(x^*)])}{T},
  \end{align*}
  \normalsize
  where $x^*$ is optimal for~\eqref{eq:1} and $x_a$ is the output of Algorithm~\ref{alg:svrg}.
\end{corollary}
We can rewrite these corollaries in terms of IFO complexity to get the following corollaries.
\begin{corollary}
  \label{cor:svrg-convex-oracle-non-const}
  If $f_i$ is convex for all $i \in [n]$, then the IFO complexity of Algorithm~\ref{alg:svrg} (with parameters from Corollary~\ref{cor:convex-non-const}) to compute an $\epsilon$-accurate solution is $O(n + (\sqrt{n}/\epsilon))$.
\end{corollary}

\begin{corollary}
  \label{cor:svrg-convex-oracle-const}
  If $f_i$ is convex for all $i \in [n]$, then the IFO complexity of Algorithm~\ref{alg:svrg} (with parameters from Corollary~\ref{cor:convex-const}) to compute $\epsilon$-accurate solution is $O(n/\epsilon)$.
\end{corollary}
These results follow from Corollary~\ref{cor:convex-non-const} and Corollary~\ref{cor:convex-const} and noting that for $m = O(n)$ the total IFO calls made by Algorithm~\ref{alg:svrg} is $O(n)$. It is instructive to quantitatively compare Corollary~\ref{cor:svrg-convex-oracle-non-const} and Corollary\ref{cor:svrg-convex-oracle-const}. With a step size independent of $n$, the convergence rate of $\svrg$ has a dependence that is in the order of $n$ (Corollary~\ref{cor:convex-const}). But this dependence can be reduced to $\sqrt{n}$ by either carefully selecting a step size that diminishes with $n$ (Corollary~\ref{cor:convex-non-const}) or by using a good initial point $x^0$ obtained by, say, running $O(n)$ iterations of $\sgd$. 

We emphasize that the convergence rate for convex case can be improved significantly by slightly modifying the algorithm (either by adding an appropriate strongly convex perturbation \citep{Xiao14} or by using a choice of $m$ that changes with epoch \citep{Zhu15}). However, it is not clear if these strategies provide any theoretical gains for the general nonconvex case.

\section{Mini-batch Nonconvex SVRG}
\label{sec:minibatch}
In this section, we study the mini-batch version of Algorithm~\ref{alg:svrg}. Mini-batching is a popular strategy, especially in multicore and distributed settings as it greatly helps one exploit parallelism and reduce the communication costs. The pseudocode for mini-batch nonconvex $\svrg$ (Algorithm~\ref{alg:minibatch-svrg}) is provided in the supplement due to lack of space. The key difference between the mini-batch \svrg and Algorithm~\ref{alg:svrg} lies in lines 6 to 8. To use mini-batches we replace line 6 with sampling (with replacement) a mini-batch $I_t \subset [n]$ of size $b$; lines 7 to 8 are replaced with the following updates:
\begin{align*}
  u_t^{s+1} &=  \tfrac{1}{|I_t|} \nlsum_{i_t \in I_t} \left(\nabla f_{i_t}(x^{s+1}_t) - \nabla f_{i_t}(\tilde{x}^{s})\right) + g^{s+1},
\\
x_{t+1}^{s+1} &= x^{s+1}_t - \eta_t u_t^{s+1}
\end{align*}
When $b=1$, this reduces to Algorithm~\ref{alg:svrg}. Mini-batch is typically used to reduce the variance of the stochastic gradient and increase the parallelism. Lemma~\ref{lem:nonconvex-minibatch-variance-lemma} (in Section~\ref{sec:key-lemmas} of the appendix) shows the reduction in the variance of stochastic gradients with mini-batch size $b$. Using this lemma, one can derive the mini-batch equivalents of Lemma~\ref{lem:nonconvex-svrg}, Theorem~\ref{thm:nonconvex-inter} and Theorem~\ref{thm:nonconvex-gen}. However, for the sake of brevity, we directly state the following main result for mini-batch $\svrg$.



\begin{theorem}
 Let $\overline{\gamma}_n$ denote the following quantity:
  \begin{align*}
    \overline{\gamma}_n := \min_{0 \leq t \leq m-1}\quad\bigl(\eta - \tfrac{\overline{c}_{t+1}\eta}{\beta} - \eta^2L - 2\overline{c}_{t+1}\eta^2\bigr).
  \end{align*}
  where $\overline{c}_m = 0$, $\overline{c}_{t} = \overline{c}_{t+1}(1 + \eta\beta + \nicefrac{2\eta^2L^2}{b} ) +  \nicefrac{\eta_t^2L^3}{b}$ for \fromto{0}{t}{m-1}. Suppose $\eta = \mu_2b/(Ln^{2/3})$ ($0 < \mu_2 < 1$), $\beta = L/n^{1/3}$, $m = \lfloor n/(3b\mu_2) \rfloor$ and $T$ is some multiple of $m$. Then for the mini-batch version of Algorithm~\ref{alg:svrg} with mini-batch size $b < n^{2/3}$, there exists universal constants $\mu_2, \nu_2 > 0$ such that we have the following: $\overline{\gamma}_n \geq \frac{\nu_2b}{Ln^{2/3}}$ and
\begin{align*}
\mathbb{E}[\|\nabla f(x_a)\|^2] &\leq \frac{Ln^{2/3} [f(x^{0}) - f(x^*)]}{bT\nu_2},
\end{align*} 
where $x^*$ is optimal for~\eqref{eq:1}.
\label{thm:nonconvex-minibatch}
\end{theorem}

It is important to compare this result with mini-batched $\sgd$. For a batch size of $b$, $\sgd$ obtains a rate of $O(1/\sqrt{bT})$ \cite{Dekel2012} (obtainable by a simple modification of Theorem~\ref{thm:sgd-conv}). Specifically, $\sgd$ has a $1/\sqrt{b}$ dependence on the batch size. In contrast, Theorem~\ref{thm:nonconvex-minibatch} shows that $\svrg$ has a much better dependence of $1/b$ on the batch size. Hence, compared to \sgd, \svrg allows more efficient mini-batching. More formally, in terms of IFO queries we have the following result. 

\begin{corollary}
\label{cor:svrg-nonconvex-oracle-minibatch}
If $f \in \Fc_n$, then the IFO complexity of the mini-batch version of Algorithm~\ref{alg:svrg} (with parameters from Theorem~\ref{thm:nonconvex-minibatch} and mini-batch size $b < n^{2/3}$) to obtain an $\epsilon$-accurate solution is $O(n + (n^{2/3}/\epsilon))$.
\end{corollary}

Corollary~\ref{cor:svrg-nonconvex-oracle-minibatch} shows an interesting property of mini-batch $\svrg$. First, note that $b$ IFO calls are required for calculating the gradient on a mini-batch of size $b$. Hence, $\svrg$ does not gain on IFO complexity by using mini-batches. However, if the $b$ gradients are calculated in parallel, then this leads to a theoretical linear speedup in multicore and distributed settings. In contrast, $\sgd$ does not yield an efficient mini-batch strategy as it requires $O(b^{1/2}/\epsilon^2)$ IFO calls for achieving an $\epsilon$-accurate solution \citep{Likdd2014}. Thus, the performance of $\sgd$ degrades with mini-batching.

\section{Comparison of the convergence rates}
\label{sec:comparison}
In this section, we give a comprehensive comparison of results obtained in this paper. In particular, we compare key aspects of the convergence rates for $\sgd$, $\gd$, and $\svrg$. The comparison is based on IFO complexity to achieve an $\epsilon$-accurate solution.

\textbf{Dependence on $n$}: The number of IFO calls of $\svrg$ and $\gd$ depend explicitly on $n$. In contrast, the number of oracle calls of $\sgd$ is independent of $n$ (Theorem~\ref{thm:sgd-conv}). However, this comes at the expense of worse dependence on $\epsilon$. The number of IFO calls in $\gd$ is proportional to $n$. But for \svrg this dependence reduces to $n^{1/2}$ for convex (Corollary~\ref{cor:convex-non-const}) and $n^{2/3}$ for nonconvex (Corollary~\ref{cor:nonconvex}) problems. Whether this difference in dependence on $n$ is due to nonconvexity or just an artifact of our analysis is an interesting open problem. 

\textbf{Dependence on $\epsilon$}: The dependence on $\epsilon$ (or alternatively $T$) follows from the convergence rates of the algorithms. $\sgd$ is seen to depend as $O(1/\epsilon^2)$ on $\epsilon$, regardless of convexity or nonconvexity. In contrast, for both convex and nonconvex settings, \svrg and \gd converge as $O(1/\epsilon)$. Furthermore, for gradient dominated functions, $\svrg$ and $\gd$ have global linear convergence. This speedup in convergence over \sgd is especially significant when medium to high accuracy solutions are required (i.e., $\epsilon$ is small).

\textbf{Assumptions used in analysis}: It is important to understand the assumptions used in deriving the convergence rates. All algorithms assume Lipschitz continuous gradients. However, $\sgd$ requires two additional subtle but important assumptions:  $\sigma$-bounded gradients and advance knowledge of $T$ (since its step sizes depend on $T$). On the other hand, both \svrg and \gd do not require these assumptions, and thus, are more flexible. 

\textbf{Step size / learning rates}: It is valuable to compare the step sizes used by the algorithms. The step sizes of \sgd shrink as the number of \emph{iterations} $T$ increases---an undesirable property. On the other hand, the step sizes of \svrg and \gd are independent of $T$. Hence, both these algorithms can be executed with a fixed step size. However, \svrg uses step sizes that depend on $n$ (see Corollary~\ref{cor:nonconvex} and Corollary~\ref{cor:convex-non-const}). A step size independent of $n$ can be used for $\svrg$ for convex $f$, albeit at cost of worse dependence on $n$ (Corollary~\ref{cor:convex-const}). $\gd$ does not have this issue as its  step size is independent of both $n$ and $T$.

\textbf{Dependence on initial point and mini-batch}: \svrg is more sensitive to the initial point in comparison to \sgd. This can be seen by comparing Corollary~\ref{cor:nonconvex} (of \svrg) to Theorem~\ref{thm:sgd-conv} (of \sgd). Hence, it is important to use a good initial point for $\svrg$. Similarly, a good mini-batch can be beneficial to \svrg. Moreover, mini-batches not only provides parallelism but also good theoretical guarantees (see Theorem~\ref{thm:nonconvex-minibatch}). In contrast, the performance gain in \sgd with mini-batches is not very pronounced (see Section~\ref{sec:minibatch}).

\section{Best of two worlds}
\label{sec:best.two}
We have seen in the previous section that \svrg combines the benefits of both \gd and \sgd. We now show that these benefits of $\svrg$ can be made more pronounced by an appropriate step size under additional assumptions.
In this case, the IFO complexity of \svrg is lower than those of \sgd and \gd. This variant of $\svrg$ ($\msvrg$) chooses a step size based on the total number of iterations $T$ (or alternatively $\epsilon$). For our discussion below, we assume that $T > n$.


\begin{theorem}
\label{thm:nonconvex-svrg-best}
Let $f \in \Fc_n$ have $\sigma$-bounded gradients. Let $\eta_t = \eta = \max\{\nicefrac{c}{\sqrt{T}},\nicefrac{\mu_1}{(Ln^{2/3})}\}$ ($\mu_1$ is the universal constant from Corollary~\ref{cor:nonconvex}), $m = \lfloor\nicefrac{n}{(3\mu_1)}\rfloor$, and $c = \sqrt{\frac{f(x^0) - f(x^*)}{2L\sigma^2}}$. Further, let 
$T$ be a multiple of $m$, $p_{m} = 1$, and $p_{i} = 0$ for $0{\;\leq\;}i{\;<\;}m$. Then, the output $x_a$ of Algorithm~\ref{alg:svrg} satisfies
\begin{align*}
& \mathbb{E}[\|\nabla f(x_a)\|^2] \\
&\leq \bar{\nu} \min\Big\lbrace 2\sqrt{\frac{2(f(x^0) - f(x^{*}))L}{T}} \sigma, \frac{ L n^{2/3} [f(x^{0}) - f(x^*)]}{T\nu_1}\Big\rbrace,
\end{align*}
where $\bar{\nu}$ is a universal constant, $\nu_1$ is the universal constant from Corollary~\ref{cor:nonconvex} and $x^*$ is an optimal solution to~\eqref{eq:1}.
\end{theorem}

\begin{corollary}
\label{cor:svrg-nonconvex-oracle-best}
If $f \in \Fc_n$ has $\sigma$-bounded gradients, the IFO complexity of Algorithm~\ref{alg:svrg} (with parameters from Theorem~\ref{thm:nonconvex-svrg-best}) to achieve an $\epsilon$-accurate solution is $O(\min\{1/\epsilon^2, n^{2/3}/\epsilon\})$.
\end{corollary}

An almost identical reasoning can be applied when $f$ is convex to get the bounds specified in Table~\ref{tab:algo-comparison}. Hence, we omit the details and  directly state the following result. 
\begin{corollary}
\label{cor:svrg-convex-oracle-best}
Suppose $f_i$ is convex for $i \in [n]$ and $f$ has $\sigma$-bounded gradients, then the IFO complexity of Algorithm~\ref{alg:svrg} (with step size $\eta = \max\{1/(L\sqrt{T}),1/(8L\sqrt{n})\}$, $m = n$ and $p_i = 1/m$ for $0 \leq i \leq m-1$ and $p_m = 0$) to achieve an $\epsilon$-accurate solution is $O(\min\{1/\epsilon^2, \sqrt{n}/\epsilon\})$.
\end{corollary}
 $\msvrg$ has a convergence rate faster than those of both $\sgd$ and $\svrg$, though this benefit is not without cost. $\msvrg$, in contrast to $\svrg$, uses the additional assumption of $\sigma$-bounded gradients. Furthermore, its step size is \emph{not} fixed since it depends on the number of iterations $T$. 
While it is often difficult in practice to compute the step size of $\msvrg$ (Theorem~\ref{thm:nonconvex-svrg-best}), it is typical to try multiple step sizes and choose the one with the best results.

\section{Experiments}
\label{sec:experiments}
We present our empirical results in this section. For our experiments, we study the problem of multiclass classification using neural networks. This is a typical nonconvex problem encountered in machine learning.

\textbf{Experimental Setup.} We train neural networks with one fully-connected hidden layer of 100 nodes and 10 softmax output nodes. We use $\ell_2$-regularization for training. We use CIFAR-10\footnote{\url{www.cs.toronto.edu/~kriz/cifar.html}}, MNIST\footnote{\url{http://yann.lecun.com/exdb/mnist/}}, and STL-10\footnote{\url{https://cs.stanford.edu/~acoates/stl10/}} datasets for our experiments. These datasets are standard in the neural networks literature. The $\ell_2$ regularization is 1e-3 for CIFAR-10 and MNIST, and 1e-2 for STL-10. The features in the datasets are normalized to the interval $[0,1]$. All the datasets come with a predefined split into training and test datasets.

\begin{figure*}[!t]
\centering
   \begin{minipage}[b]{.28\textwidth}
   \includegraphics[width=\textwidth]{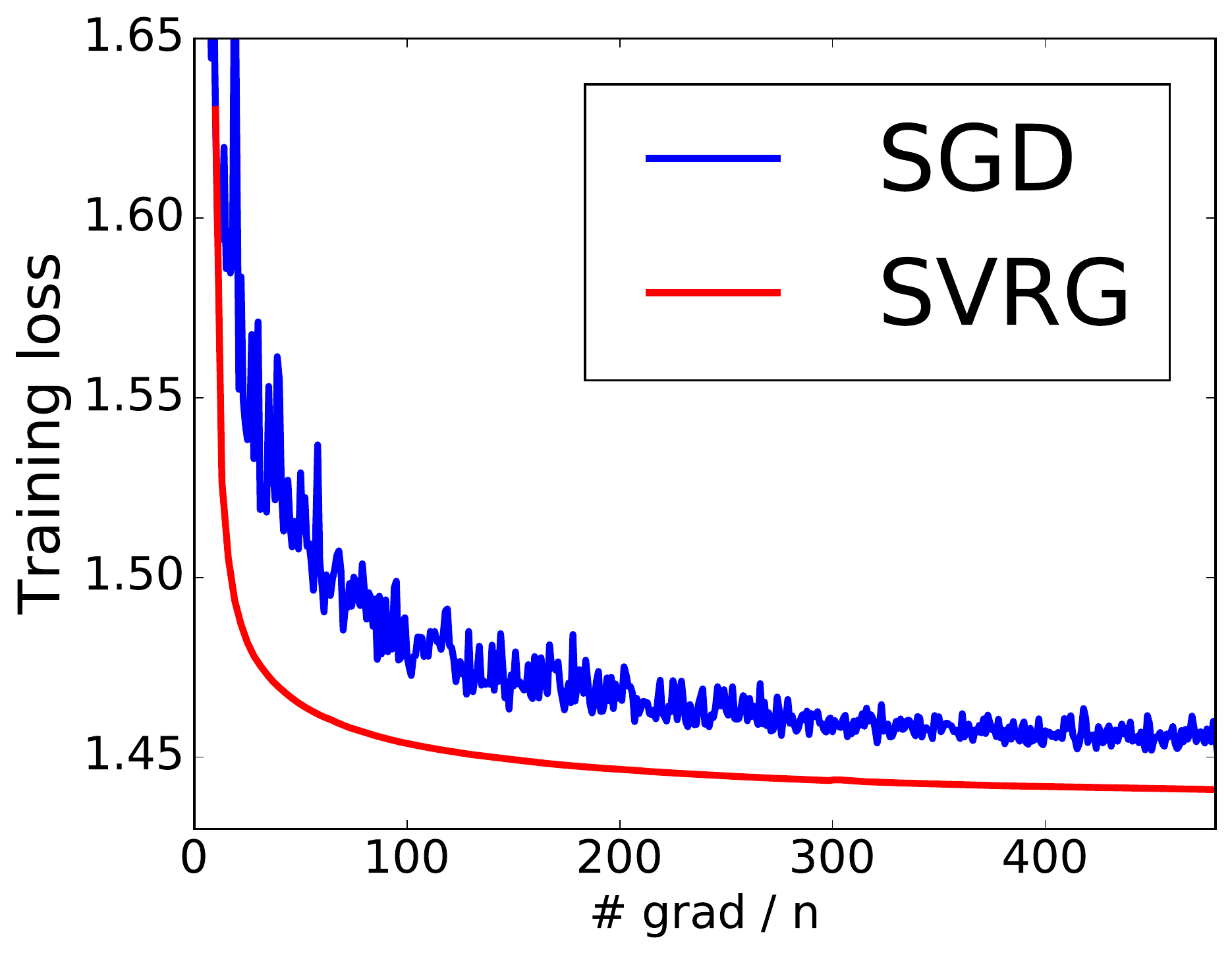}
   \end{minipage} %
   \begin{minipage}[b]{.28\textwidth}
   \includegraphics[width=\textwidth]{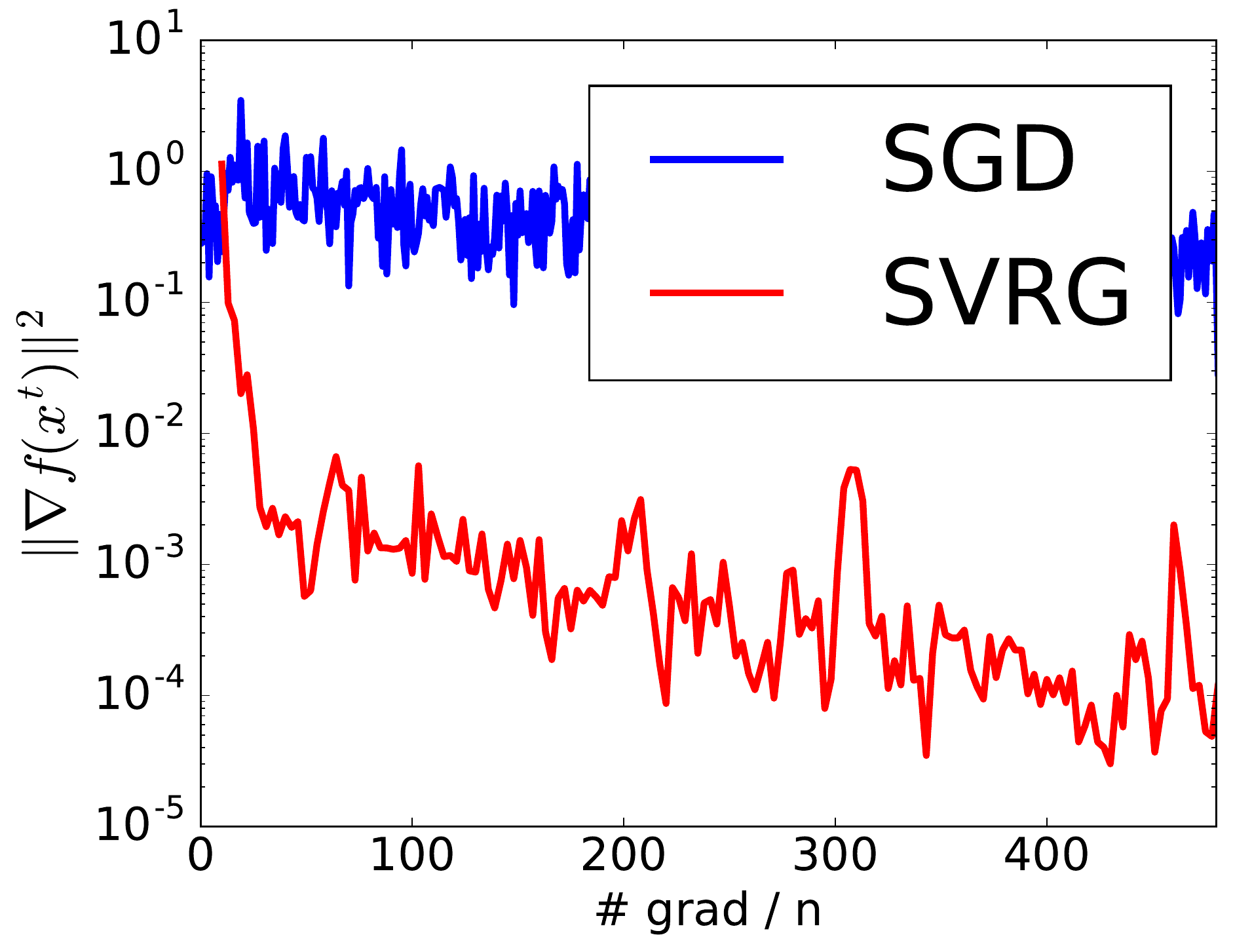}
   \end{minipage}
   \begin{minipage}[b]{.28\textwidth}
   \includegraphics[width=\textwidth]{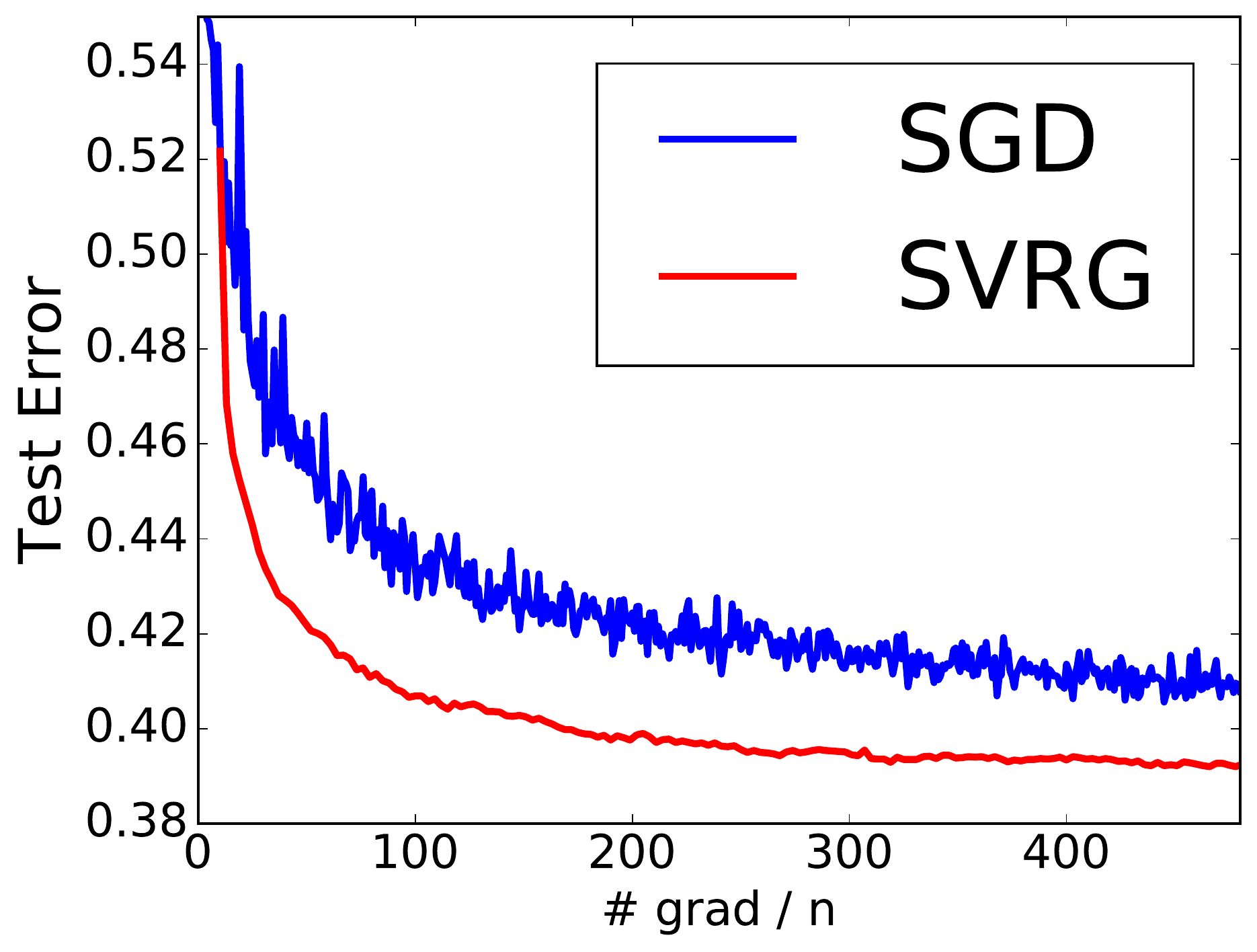}
   \end{minipage}
   \begin{minipage}[b]{.28\textwidth}
   \includegraphics[width=\textwidth]{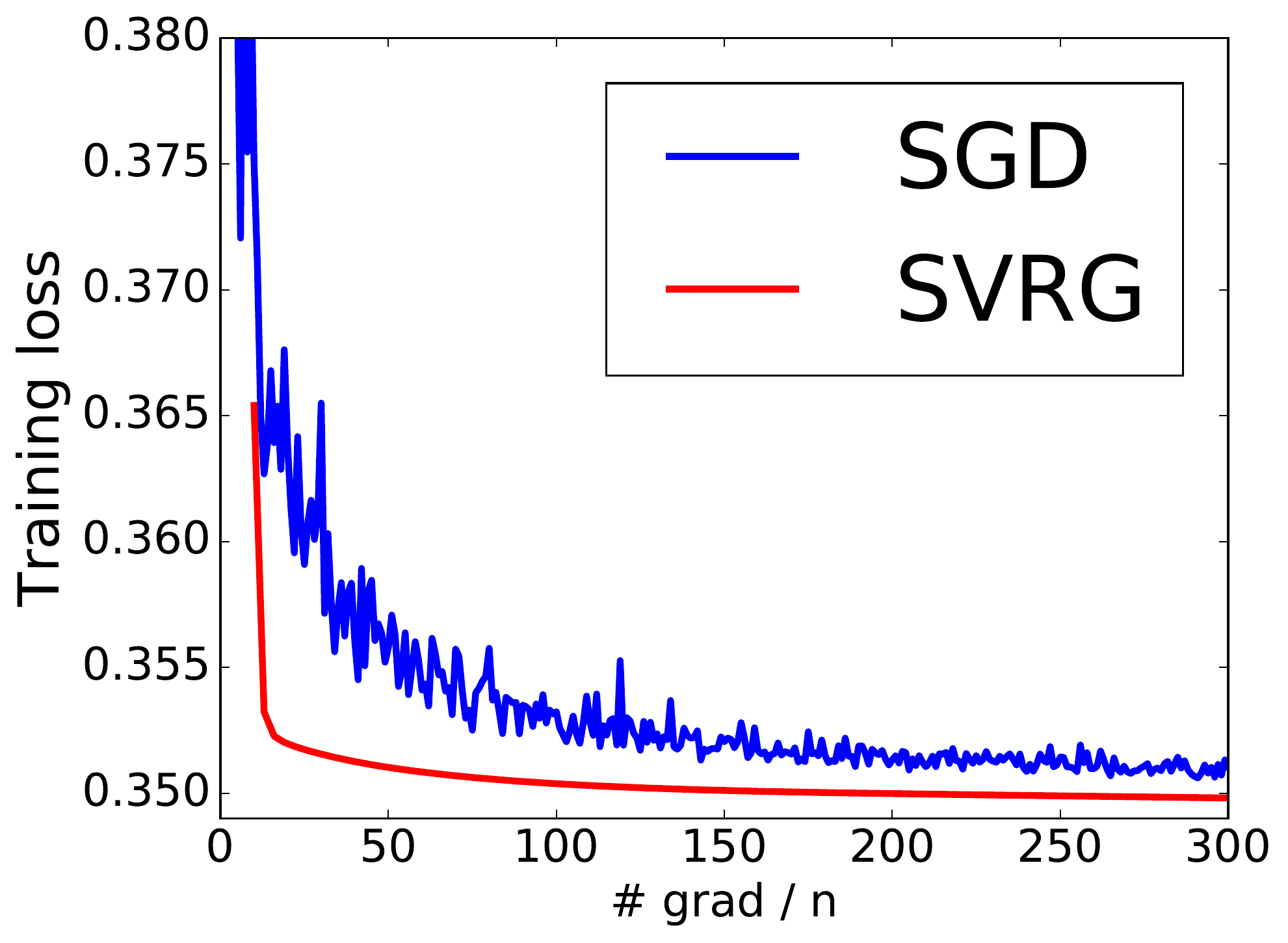}
   \end{minipage} %
   \begin{minipage}[b]{.28\textwidth}
   \includegraphics[width=\textwidth]{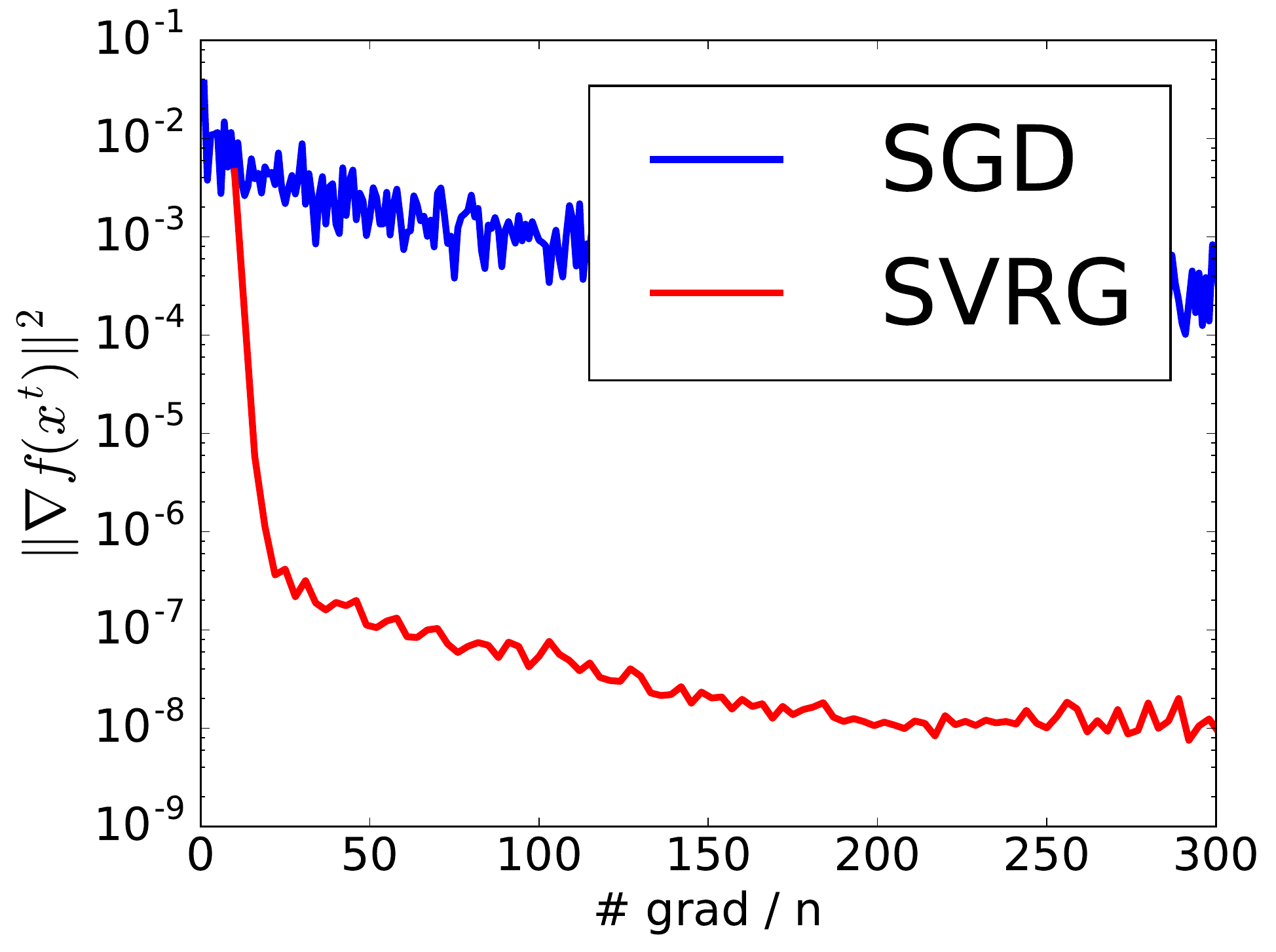}
   \end{minipage}
   \begin{minipage}[b]{.28\textwidth}
   \includegraphics[width=\textwidth]{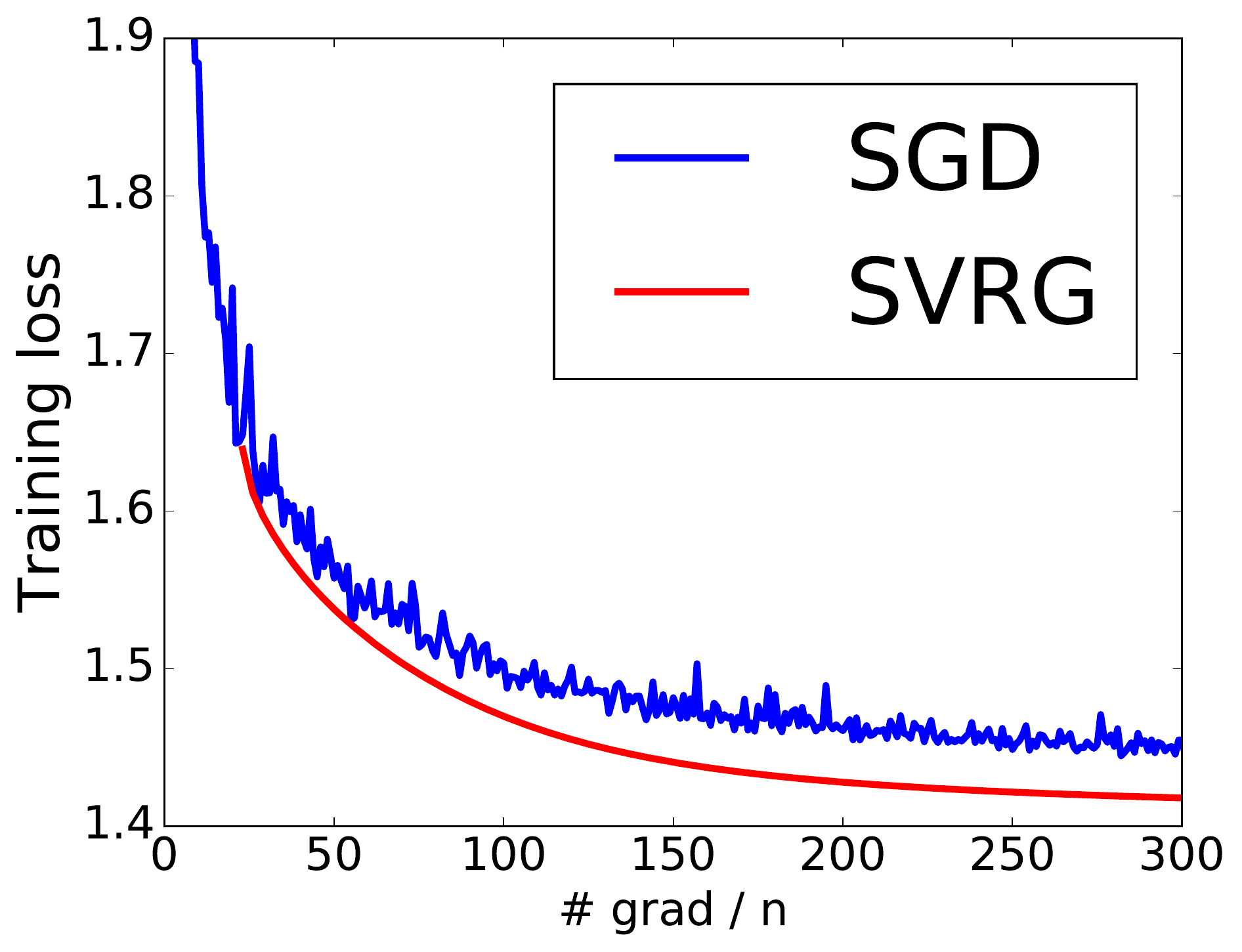}
   \end{minipage}
	\caption{\small Neural network results for CIFAR-10, MNIST and STL-10 datasets. The top row represents the results for CIFAR-10 dataset. The bottom left and middle figures represent the results for MNIST dataset. The bottom right figure represents the result for STL-10.}
	\label{fig:results}
\end{figure*}

We compare $\sgd$ (the \emph{de-facto} algorithm for training neural networks) against nonconvex \svrg. The step size (or learning rate) is critical for \sgd. We set the learning rate of \sgd using the popular $t-$inverse schedule $\eta_t = \eta_0(1 + \eta'\lfloor t/n \rfloor)^{-1}$, where $\eta_0$ and $\eta'$ are chosen so that \sgd gives the best performance on the training loss. In our experiments, we also use $\eta' = 0$; this results in a fixed step size for $\sgd$. For \svrg, we use a fixed step size as suggested by our analysis. Again, the step size is chosen so that \svrg gives the best performance on the training loss.

\textbf{Initialization \& mini-batching.} Initialization is critical to training of neural networks. We use the normalized initialization in \citep{Glorot10} where parameters are chosen uniformly from $[-\sqrt{6/(n_i + n_o)}$, $\sqrt{6/(n_i + n_o)}]$ where $n_i$ and $n_o$ are the number of input and output layers of the neural network, respectively.

For \svrg, we use $n$ iterations  of $\sgd$ for CIFAR-10 and MINST and $2n$ iterations of $\sgd$  for STL-10 before running  Algorithm~\ref{alg:svrg}. Such initialization is standard for variance reduced schemes even for convex problems \citep{Johnson13,Schmidt13}. As noted  earlier in Section~\ref{sec:comparison}, $\svrg$ is more sensitive than \sgd to the initial point, so such an initialization is typically helpful. We use mini-batches of size 10 in our experiments. $\sgd$ with mini-batches is common in training neural networks. Note that mini-batch training is especially beneficial for $\svrg$, as shown by our analysis in Section~\ref{sec:minibatch}. Along the lines of theoretical analysis provided by Theorem~\ref{thm:nonconvex-minibatch}, we use an epoch size $m = n/10$ in our experiments. 

\textbf{Results.} We report objective function (training loss), test error (classification error on the test set), and $\|\nabla f(x^t)\|^2$ (convergence criterion throughout our analysis) for the datasets. For all the algorithms, we compare these criteria against the number of \emph{effective passes} through the data, i.e., IFO calls divided by $n$. This includes the cost of calculating the full gradient at the end of each epoch of $\svrg$. Due to the $\sgd$ initialization in $\svrg$ and mini-batching, the $\svrg$ plots start from x-axis value of 10 for CIFAR-10 and MNIST and 20 for STL-10. Figure~\ref{fig:results} shows the results for our experiment. It can be seen that the $\|\nabla f(x^t)\|^2$ for $\svrg$ is lower compared to $\sgd$, suggesting faster convergence to a stationary point. Furthermore, the training loss is also lower compared to $\sgd$ in all the datasets. Notably, the test error for CIFAR-10 is lower for $\svrg$, indicating better generalization; we did not notice substantial difference in test error for MNIST and STL-10 (see Section~\ref{sec:remaining-expts} in the appendix). Overall, these results on a network with one hidden layer are promising; it will be interesting to study \svrg for deep neural networks in the future.
\vspace*{-3pt}
\section{Discussion}
In this paper, we examined a VR scheme for nonconvex optimization. We showed that by employing VR in stochastic methods, one can perform better than both $\sgd$ and $\gd$ in the context of nonconvex optimization. When the function $f$ in~\eqref{eq:1} is gradient dominated, we proposed a variant of \svrg that has linear convergence to the \emph{global} minimum. Our analysis shows that \svrg has a number of interesting properties that include convergence with fixed step size, descent property after every epoch; a property that need not hold for $\sgd$. We also showed that $\svrg$, in contrast to $\sgd$, enjoys efficient mini-batching, attaining speedups linear in the size of the mini-batches in parallel settings. Our analysis also reveals that the initial point and use of mini-batches are important to $\svrg$.

Before concluding the paper, we would like to discuss the implications of our work and few caveats. One should exercise some caution while interpreting the results in the paper. All our theoretical results are based on the stationarity gap. In general, this does not necessarily translate to optimality gap or low training loss and test error. One criticism against VR schemes in nonconvex optimization is the general wisdom that variance in the stochastic gradients of $\sgd$ can actually help it escape local minimum and saddle points. In fact, \citet{Ge15} add additional noise to the stochastic gradient in order to escape saddle points. However, one can reap the benefit of VR schemes even in such scenarios. For example, one can envision an algorithm which uses $\sgd$ as an exploration tool to obtain a good initial point and then uses a VR algorithm as an exploitation tool to quickly converge to a \emph{good} local minimum. In either case, we believe variance reduction can be used as an important tool alongside other tools like momentum, adaptive learning rates for faster and better nonconvex optimization.

\bibliographystyle{custom}
\setlength{\bibsep}{3pt}
\bibliography{bibfile}

\appendix
\input{appendix}

\end{document}

%% file: appendix.tex
\section*{Appendix}

\section{Nonconvex SGD: Convergence Rate}

\section*{Proof of Theorem~\ref{thm:sgd-conv}}
\begin{theorem*}
  Suppose $f$ has $\sigma$-bounded gradient; let $\eta_t = \eta = c/\sqrt{T}$ where $c = \sqrt{\tfrac{2(f(x^0) - f(x^*))}{L\sigma^2}}$, 
  and $x^*$ is an optimal solution to~\eqref{eq:1}. Then, the iterates of Algorithm~\ref{alg:sgd} satisfy
  \begin{align*}
    \min_{0 \leq t \leq T-1} \mathbb{E}[\|\nabla f(x^t)\|^2] \leq
    \sqrt{\tfrac{2(f(x^0) - f(x^*)) L}{T}}\sigma.
  \end{align*}
\end{theorem*}
\begin{proof}
  We include the proof here for completeness. Please refer to \citep{Ghadimi13} for a more general result.

The iterates of Algorithm~\ref{alg:sgd} satisfy the following bound:
\begin{align}
  &\E[f(x^{t+1})] \leq \mathbb{E}[f(x^t) + \left\langle \nabla f(x^t), x^{t+1} - x^{t} \right\rangle \nonumber \\
  & \qquad \qquad + \tfrac{L}{2} \|x^{t+1} - x^t \|^2] \\
  & \leq \mathbb{E}[f(x^t)] - \eta_t \mathbb{E}[\|\nabla f(x^t)\|^2] +\tfrac{L\eta_t^2}{2} \mathbb{E}[\|\nabla f_{i_t}(x^t)\|^2] \nonumber \\
  & \leq \mathbb{E}[f(x^t)] - \eta_t \mathbb{E}[\|\nabla f(x^t)\|^2] +\tfrac{L\eta_t^2}{2} \sigma^2.
    \label{eq:sgd-proof-eq}
\end{align}
The first inequality follows from Lipschitz continuity of $\nabla f$. The second inequality follows from the update in Algorithm~\ref{alg:sgd} and since $\E_{i_t}[\nabla f_{i_t}(x^t)] = \nabla f(x^t)$ (unbiasedness of the stochastic gradient). The last step uses our assumption on gradient boundedness. Rearranging Equation~\eqref{eq:sgd-proof-eq} we obtain
\begin{align}
  \E[\|\nabla f(x^t)\|^2] \leq \tfrac{1}{\eta_t}\mathbb{E}[f(x^t) - f(x^{t+1})] +\tfrac{L\eta_t}{2} \sigma^2.
  \label{sgd-proof-eq2}
\end{align}
Summing Equation~\eqref{sgd-proof-eq2} from $t=0$ to $T-1$ and using that $\eta_t$ is constant $\eta$ we obtain
\begin{align*}
  \min_{t} \E[\|\nabla f(x^t)\|^2] &\leq \tfrac{1}{T}\nlsum_{t=0}^{T-1} \mathbb{E}[\|f(x^t)\|^2] \\
  &\leq \tfrac{1}{T\eta}\mathbb{E}[f(x^0) - f(x^{T})] +\tfrac{L\eta}{2} \sigma^2 \\
  &\leq \tfrac{1}{T\eta} (f(x^0) - f(x^{*})) +\tfrac{L\eta}{2} \sigma^2 \\
  &\leq \tfrac{1}{\sqrt{T}}\Bigl(\tfrac{1}{c}\bigl(f(x^0) - f(x^{*})\bigr) +\tfrac{Lc}{2} \sigma^2\Bigr).
\end{align*}
The first step holds because the minimum is less than the average. The second and third steps are obtained from Equation~\eqref{sgd-proof-eq2} and the fact that $f(x^*) \leq f(x^T)$, respectively. The final inequality follows upon using $\eta=c/\sqrt{T}$. By setting 
$$
c = \sqrt{\frac{2(f(x^0) - f(x^*))}{L\sigma^2}}
$$
in the above inequality, we get the desired result.
\end{proof}

\section{Nonconvex SVRG}

In this section, we provide the proofs of the results for nonconvex $\svrg$. We first start with few useful lemmas and then proceed towards the main results.

\begin{lemma}
\label{lem:nonconvex-svrg}
For $c_t, c_{t+1}, \beta_t > 0$, suppose we have 
$$
c_{t} = c_{t+1}(1 + \eta_t\beta_t + 2\eta_t^2L^2 ) +  \eta_t^2L^3.
$$ 
Let $\eta_t$, $\beta_t$ and $c_{t+1}$ be chosen such that $\Gamma_t > 0$ (in Equation~\eqref{eq:Gamma-t}). The iterate $x^{s+1}_t$ in Algorithm~\ref{alg:svrg} satisfy the bound:
  \begin{align*}
    \mathbb{E}[\|\nabla f(x^{s+1}_{t})\|^2] \leq \frac{R^{s+1}_{t} - R^{s+1}_{t+1}}{\Gamma_t},
  \end{align*}
  where $R^{s+1}_{t} := \mathbb{E}[f(x^{s+1}_{t}) + c_{t} \|x^{s+1}_{t} - \tilde{x}^{s}\|^2]$ for $0 \leq s \leq S-1$.
\end{lemma}
\begin{proof}
Since $f$ is $L$-smooth we have
\begin{align*}
&\mathbb{E}[f(x^{s+1}_{t+1})] \leq \mathbb{E}[f(x^{s+1}_{t}) + \langle \nabla f(x^{s+1}_t), x^{s+1}_{t+1} - x^{s+1}_t \rangle \nonumber \\
& \qquad \qquad \qquad \qquad \qquad + \tfrac{L}{2} \| x^{s+1}_{t+1} - x^{s+1}_t \|^2].
\end{align*}
Using the \svrg update in Algorithm~\ref{alg:svrg} and its unbiasedness, the right hand side above is further upper bounded by
\begin{align}
  \mathbb{E}[f(x^{s+1}_{t}) - \eta_t \|\nabla f(x^{s+1}_{t})\|^2 + \tfrac{L\eta_t^2}{2} \|v^{s+1}_t \|^2].
\label{eq:svrg-proof-eq1}
\end{align}
Consider now the Lyapunov function
$$
R^{s+1}_{t} := \mathbb{E}[f(x^{s+1}_{t}) + c_{t} \|x^{s+1}_{t} - \tilde{x}^s\|^2].
$$
For bounding it we will require the following:
\begin{align}
&\mathbb{E}[\|x^{s+1}_{t+1} - \tilde{x}^s\|^2] = \mathbb{E}[\|x^{s+1}_{t+1} - x^{s+1}_t + x^{s+1}_t - \tilde{x}^s\|^2] \nonumber \\
&= \mathbb{E}[\|x^{s+1}_{t+1} - x^{s+1}_t\|^2 + \|x^{s+1}_t - \tilde{x}^s\|^2 \nonumber \\
& \qquad \qquad \qquad + 2\langle x^{s+1}_{t+1} - x^{s+1}_t, x^{s+1}_t - \tilde{x}^s\rangle] \nonumber \\
&= \mathbb{E}[\eta_t^2\|v^{s+1}_t\|^2 + \|x^{s+1}_t - \tilde{x}^s\|^2] \nonumber \\
& \nonumber\quad \qquad \qquad - 2\eta_t \mathbb{E}[\langle \nabla f(x^{s+1}_t), x^{s+1}_t - \tilde{x}^s\rangle]\\
&\leq \mathbb{E}[\eta_t^2\|v^{s+1}_t\|^2 + \|x^{s+1}_t - \tilde{x}^s\|^2] \nonumber \\
& \qquad + 2 \eta_t \mathbb{E}\left[\tfrac{1}{2\beta_t} \|\nabla f(x^{s+1}_t)\|^2 + \tfrac{1}{2}\beta_t \| x^{s+1}_t - \tilde{x}^s \|^2 \right].
\label{eq:svrg-proof-eq2}
\end{align}
The second equality follows from the unbiasedness of the update of $\svrg$. The last inequality follows from a simple application of Cauchy-Schwarz and Young's inequality. Plugging Equation~\eqref{eq:svrg-proof-eq1} and Equation~\eqref{eq:svrg-proof-eq2} into $R^{s+1}_{t+1}$, we obtain the following bound:
\begin{align}
R^{s+1}_{t+1} &\leq \mathbb{E}[f(x^{s+1}_{t}) - \eta_t \|\nabla f(x^{s+1}_{t})\|^2 + \tfrac{L\eta_t^2}{2} \|v^{s+1}_t \|^2] \nonumber \\
&  \qquad + \mathbb{E}[c_{t+1}\eta_t^2\|v^{s+1}_t\|^2 + c_{t+1}\|x^{s+1}_t - \tilde{x}^s\|^2] \nonumber \\
&  \qquad + 2 c_{t+1}\eta_t \mathbb{E}\left[\tfrac{1}{2\beta_t} \|\nabla f(x^{s+1}_t)\|^2 + \tfrac{1}{2}\beta_t \| x^{s+1}_t - \tilde{x}^s \|^2 \right] \nonumber\\
&\leq \mathbb{E}[f(x^{s+1}_{t}) - \left(\eta_t - \tfrac{c_{t+1}\eta_t}{\beta_t}\right) \|\nabla f(x^{s+1}_{t})\|^2 \nonumber\\
&  \qquad + \left(\tfrac{L\eta_t^2}{2} + c_{t+1}\eta_t^2 \right)\mathbb{E}[\|v^{s+1}_t\|^2] \nonumber\\
&  \qquad + \left( c_{t+1} + c_{t+1}\eta_t\beta_t \right) \mathbb{E}\left[\| x^{s+1}_t - \tilde{x}^s \|^2 \right].
\label{eq:svrg-proof-eq3}
\end{align}
To further bound this quantity, we use Lemma~\ref{lem:nonconvex-variance-lemma} to bound $\mathbb{E}[\|v^{s+1}_{t}\|^2]$, so that upon substituting it in Equation~\eqref{eq:svrg-proof-eq3}, we see that
\begin{align*}
& R^{s+1}_{t+1} \leq \mathbb{E}[f(x^{s+1}_{t})] \nonumber \\
& \qquad - \left(\eta_t - \tfrac{c_{t+1}\eta_t}{\beta_t} - \eta_t^2L - 2c_{t+1}\eta_t^2\right) \mathbb{E}[\|\nabla f(x^{s+1}_{t})\|^2] \nonumber\\
& \qquad + \left[c_{t+1}\bigl(1 + \eta_t\beta_t + 2\eta_t^2L^2\bigr)+\eta_t^2L^3\right] 
  \mathbb{E}\left[\| x^{s+1}_t - \tilde{x}^s \|^2 \right] \nonumber \\
& \leq R^{s+1}_{t} - \bigl(\eta_t - \tfrac{c_{t+1}\eta_t}{\beta_t} - \eta_t^2L - 2c_{t+1}\eta_t^2\bigr) \mathbb{E}[\|\nabla f(x^{s+1}_{t})\|^2].
\end{align*}
The second inequality follows from the definition of $c_{t}$ and $R_t^{s+1}$, thus concluding the proof. 
\end{proof}


\section*{Proof of Theorem~\ref{thm:nonconvex-inter}}
\begin{theorem*}
  Let $f \in \Fc_n$. Let $c_m = 0$, $\eta_t = \eta > 0$, $\beta_t = \beta > 0$, and $c_{t} = c_{t+1}(1 + \eta\beta + 2\eta^2L^2 ) +  \eta^2L^3$ such that $\Gamma_t > 0$ for \fromto{0}{t}{m-1}. Define the quantity $\gamma_n := \min_t \Gamma_t$. 
  Further, let $p_{i} = 0$ for $0{\;\leq\;}i{\;<\;}m$ and $p_{m} = 1$, and let $T$ be a multiple of $m$. Then for the output $x_a$ of Algorithm~\ref{alg:svrg} we have
  \begin{align*}
    \mathbb{E}[\|\nabla f(x_a)\|^2] \leq \frac{f(x^{0}) - f(x^*)}{T\gamma_n},
  \end{align*}
  where $x^*$ is an optimal solution to~\eqref{eq:1}.
\end{theorem*}
\begin{proof}
  Since $\eta_t = \eta$ for $t \in \{0, \dots, m-1\}$, using  Lemma~\ref{lem:nonconvex-svrg} and telescoping the sum, we obtain
\begin{align*}
  \nlsum_{t=0}^{m-1} \mathbb{E}[\|\nabla f(x^{s+1}_{t})\|^2] \leq \frac{R^{s+1}_{0} - R^{s+1}_{m}}{\gamma_n}.
\end{align*}
This inequality in turn implies that
\begin{align}
  \label{eq:descent-property}
  \nlsum_{t=0}^{m-1} \mathbb{E}[\|\nabla f(x^{s+1}_{t})\|^2] \leq \frac{\mathbb{E}[f(\tilde{x}^s) - f(\tilde{x}^{s+1})]}{\gamma_n},
\end{align}
where we used that $R^{s+1}_{m} = \mathbb{E}[f(x^{s+1}_m)] = \mathbb{E}[f(\tilde{x}^{s+1})]$ (since $c_m = 0$, $p_{m} = 1$, and $p_i = 0$ for $i < m$), and that $R^{s+1}_{0} = \mathbb{E}[f(\tilde{x}^s)]$ (since $x^{s+1}_0 = \tilde{x}^s$, as $p_{m} = 1$ and $p_i = 0$ for $i < m$). Now sum over all epochs to obtain 
\begin{align}
  \frac{1}{T} \sum_{s=0}^{S-1}\sum_{t=0}^{m-1} \mathbb{E}[\|\nabla f(x^{s+1}_{t})\|^2] \leq \frac{f(x^{0}) - f(x^*)}{T\gamma_n}.
\label{eq:nonconvex-cor-eq1}
\end{align}
The above inequality used the fact that $\tilde{x}^0 = x^0$. Using the above inequality and the definition of $x_a$ in Algorithm~\ref{alg:svrg}, we obtain the desired result.
\end{proof}

\section*{Proof of Theorem~\ref{thm:nonconvex-gen}}
\begin{theorem*}
  Suppose $f \in \Fc_n$. Let $\eta = \mu_0/(Ln^{\alpha})$ ($0 < \mu_0 < 1$ and $0 < \alpha \leq 1$), $\beta = L/n^{\alpha/2}$, $m = \lfloor n^{3\alpha/2}/(3\mu_0) \rfloor$ and $T$ is some multiple of $m$. Then there exists universal constants $\mu_0, \nu > 0$ such that we have the following: $\gamma_n \geq \frac{\nu}{Ln^{\alpha}}$ in Theorem~\ref{thm:nonconvex-inter} and
  \begin{align*}
    \mathbb{E}[\|\nabla f(x_a)\|^2] &\leq \frac{Ln^{\alpha} [f(x^{0}) - f(x^*)]}{T\nu},
  \end{align*} 
  where $x^*$ is an optimal solution to the problem in~\eqref{eq:1} and $x_a$ is the output of Algorithm~\ref{alg:svrg}.
\end{theorem*}
\begin{proof}
  For our analysis, we will require an upper bound on $c_{0}$. We observe that $c_0 = \tfrac{\mu_0^2L}{n^{2\alpha}} \tfrac{(1 + \theta)^m - 1}{\theta}$ where $\theta = 2\eta^2L^2 + \eta\beta$. This is obtained using the relation $c_{t} = c_{t+1}(1 + \eta\beta + 2\eta^2L^2 ) +  \eta^2L^3$ and the fact that $c_m = 0$. Using the specified values of $\beta$ and $\eta$ we have
  \begin{align*}
    \theta = 2\eta^2L^2 + \eta\beta = \frac{2\mu_0^2}{n^{2\alpha}} + \frac{\mu_0}{n^{3\alpha/2}} \leq \frac{3\mu_0}{n^{3\alpha/2}}.
  \end{align*}
  The above inequality follows since $\mu_0 \leq 1$ and $n \geq 1$. Using the above bound on $\theta$, we get 
  \begin{align}
    \label{eq:c0-bound}
    c_0 &= \frac{\mu_0^2L}{n^{2\alpha}} \frac{(1 + \theta)^m - 1}{\theta} = \frac{\mu_0L ((1 + \theta)^m - 1)}{2\mu_0 + n^{\alpha/2}} \nonumber \\
        &\leq \frac{\mu_0L ((1 + \frac{3\mu_0}{n^{3\alpha/2}})^{\lfloor \nicefrac{n^{3\alpha/2}}{3\mu_0} \rfloor} - 1)}{2\mu_0 + n^{\alpha/2}} \nonumber \\
        &\leq n^{-\alpha/2}(\mu_0L (e - 1)),
  \end{align}
  wherein the second inequality follows upon noting that $(1 + \frac{1}{l})^l$ is increasing for $l>0$ and $\lim_{l \rightarrow \infty} (1 + \frac{1}{l})^l = e$ (here $e$ is the Euler's number). Now we can lower bound $\gamma_n$, as
  \begin{align*}
    \gamma_n &= \min_t \bigl(\eta - \tfrac{c_{t+1}\eta}{\beta} - \eta^2L - 2c_{t+1}\eta^2\bigr) \\
             &\geq \bigl(\eta - \tfrac{c_{0}\eta}{\beta} - \eta^2L - 2c_{0}\eta^2\bigr) \geq \frac{\nu}{Ln^{\alpha}},
  \end{align*}
  where $\nu$ is a constant independent of $n$. The first inequality holds since $c_t$ decreases with $t$. The second inequality holds since (a) $c_0/\beta$ is upper bounded by a constant independent of $n$ as $c_0/\beta \leq \mu_0(e-1)$ (follows from Equation~\eqref{eq:c0-bound}), (b) $\eta^2L \leq \mu_0 \eta$ and (c) $2c_0\eta^2 \leq 2\mu_0^2(e-1)\eta$ (follows from Equation~\eqref{eq:c0-bound}). By choosing $\mu_0$ (independent of $n$) appropriately, one can ensure that $\gamma_n \geq \nu/(Ln^{\alpha})$ for some universal constant $\nu$. For example, choosing $\mu_0 = 1/4$, we have $\gamma_n \geq \nu/(Ln^{\alpha})$ with $\nu = 1/40$. Substituting the above lower bound in Equation~\eqref{eq:nonconvex-cor-eq1}, we obtain the desired result.
\end{proof}

\section*{Proof of Corollary~\ref{cor:svrg-nonconvex-oracle-gen}}
\begin{corollary*}
  Suppose $f \in \Fc_n$. Then the IFO complexity of Algorithm~\ref{alg:svrg} (with parameters from Theorem~\ref{thm:nonconvex-gen}) for achieving an $\epsilon$-accurate solution is:
  \begin{align*}
    \text{IFO calls} =
    \begin{cases}
      O\left(n + (n^{1 - \frac{\alpha}{2}}/\epsilon)\right),  & \text{if } \alpha < 2/3, \\
      O\left(n + (n^{\alpha}/\epsilon)\right), & \text{if } \alpha \geq 2/3.
    \end{cases}
  \end{align*}
\end{corollary*}
\begin{proof}
  This result follows from Theorem~\ref{thm:nonconvex-gen} and the fact that $m = \lfloor n^{3\alpha/2}/(3\mu_0) \rfloor$. Suppose $\alpha < 2/3$, then $m = o(n)$. However, $n$ IFO calls are invested in calculating the average gradient at the end of each epoch. In other words, computation of average gradient requires $n$ IFO calls for every $m$ iterations of the algorithm. Using this relationship, we get $O\bigl(n + (n^{1 - \tfrac{\alpha}{2}}/\epsilon)\bigr)$ in this case.

  On the other hand, when $\alpha \geq 2/3$, the total number of IFO calls made by Algorithm~\ref{alg:svrg} in each epoch is $\Omega(n)$ since $m = \lfloor n^{3\alpha/2}/(3\mu_0) \rfloor$. Hence, the oracle calls required for calculating the average gradient (per epoch) is of lower order, leading to $O\bigl(n + (n^{\alpha}/\epsilon)\bigr)$ IFO calls.
\end{proof}

\section{GD-SVRG}

\section*{Proof of Theorem~\ref{thm:gd-svrg-thm1}}
\begin{theorem*}
  Suppose $f$ is $\tau$-gradient dominated where $\tau > n^{1/3}$. Then, the iterates of Algorithm~\ref{alg:gd-svrg} with $T = \lceil 2L\tau n^{2/3}/\nu_1 \rceil$, $m = \lfloor n/(3\mu_1) \rfloor$, $\eta_t = \mu_1/(Ln^{2/3})$ for all $0 \leq t \leq m-1$ and $p_{m} = 1$ and $p_i = 0$ for all $0 \leq i < m$ satisfy
  $$\mathbb{E}[\|\nabla f(x^k)\|^2] \leq 2^{-k}[\| \nabla f(x^{0}) \|^2].
  $$
  Here $\mu_1$ and $\nu_1$ are the constants used in Corollary~\ref{cor:nonconvex}.
\end{theorem*}
\begin{proof}
  Corollary~\ref{cor:nonconvex} shows that the iterates of Algorithm~\ref{alg:gd-svrg} satisfy
\begin{align*}
\mathbb{E}[\| \nabla f(x^k) \|^2] &\leq \frac{Ln^{2/3} \mathbb{E}[f(x^{k-1}) - f(x^*)]}{T\nu_1}.
\end{align*}
Substituting the specified value of $T$ in the above inequality, we have
\begin{align*}
\mathbb{E}[\| \nabla f(x^k) \|^2] &\leq 
  \frac{1}{2\tau}\bigl(\mathbb{E}[f(x^{k-1}) - f(x^*)]\bigr) \\ &\leq\tfrac{1}{2}\mathbb{E}[\| \nabla f(x^{k-1})\|^2].
\end{align*}
The second inequality follows from $\tau$-gradient dominance of the function $f$.
\end{proof}

\section*{Proof of Theorem~\ref{thm:gd-svrg-thm2}}
\begin{theorem*}
  If $f$ is $\tau$-gradient dominated ($\tau > n^{1/3}$), then with $T = \lceil 2L\tau n^{2/3}/\nu_1 \rceil$, $m = \lfloor n/(3\mu_1) \rfloor$, $\eta_t = \mu_1/(Ln^{2/3})$ for $0 \leq t \leq m-1$ and $p_{m} = 1$ and $p_i = 0$ for all $0 \leq i < m$, the iterates of Algorithm~\ref{alg:gd-svrg} satisfy
  $$\mathbb{E}[f(x^k) - f(x^*)] \leq 2^{-k}[f(x^0) - f(x^*)].$$
  Here $\mu_1$, $\nu_1$ are as in Corollary~\ref{cor:nonconvex}; $x^*$ is an optimal solution.
\end{theorem*}
\begin{proof}
  The proof mimics that of Theorem~\ref{thm:gd-svrg-thm1}; now we have the following condition on the iterates of Algorithm~\ref{alg:gd-svrg}:
  \begin{align}
    \label{eq:gd-svrg-thm2-eq1}
    \mathbb{E}[\| \nabla f(x^k) \|^2] \leq \frac{\mathbb{E}[f(x^{k-1}) - f(x^*)]}{2\tau}.
  \end{align}
  However, $f$ is $\tau$-gradient dominated, so $\mathbb{E}[\| \nabla f(x^k) \|^2] \geq  \mathbb{E}[f(x^k) - f(x^*)]/\tau$, which combined with Equation~\eqref{eq:gd-svrg-thm2-eq1} concludes the proof.
\end{proof}

\section{Convex SVRG: Convergence Rate}

\section*{Proof of Theorem~\ref{thm:convex}}
\begin{theorem*}
  If $f_i$ is convex for all $i \in [n]$, $p_i = 1/m$ for \fromto{0}{i}{m-1}, and $p_m = 0$, then for Algorithm~\ref{alg:svrg}, we have
  \begin{align*}
    &\mathbb{E}[\|\nabla f(x_a)\|^2] \leq \frac{L\|x^{0} - x^*\|^2 + 4mL^2\eta^2 [f(x^{0}) - f(x^*)]}{T\eta(1 - 4L\eta)},
  \end{align*}
  where $x^*$ is optimal for~\eqref{eq:1} and $x_a$ is the output of Algorithm~\ref{alg:svrg}.
\end{theorem*}
\begin{proof}
Consider the following sequence of inequalities:
\begin{align*}
&\mathbb{E}[\|x^{s+1}_{t+1} - x^*\|^2] = \mathbb{E}[\|x^{s+1}_{t} - \eta v^{s+1}_t - x^*\|^2] \\
&\leq \mathbb{E}[\|x^{s+1}_{t} - x^*\|^2] + \eta^2 \mathbb{E}[\|v^{s+1}_t\|^2] \\
& \qquad \qquad - 2\eta \mathbb{E}[\langle v^{s+1}_t, x^{s+1}_t - x^* \rangle] \\
&\leq  \mathbb{E}[\|x^{s+1}_{t} - x^*\|^2] + \eta^2 \mathbb{E}[\|v^{s+1}_t\|^2] \\
& \qquad \qquad - 2\eta \mathbb{E}[f(x^{s+1}_t) - f(x^*)] \\
&\leq \mathbb{E}[\|x^{s+1}_{t} - x^*\|^2] - 2\eta(1 - 2L\eta) \mathbb{E}[f(x^{s+1}_t) - f(x^*)] \\
& \qquad \qquad + 4L\eta^2 \mathbb{E}[f(\tilde{x}^{s}) - f(x^*)] \\
&= \mathbb{E}[\|x^{s+1}_{t} - x^*\|^2] - 2\eta(1 - 4L\eta) \mathbb{E}[f(x^{s+1}_t) - f(x^*)] \\
& \qquad \qquad + 4L\eta^2 \mathbb{E}[f(\tilde{x}^{s}) - f(x^*)] - 4L\eta^2 \mathbb{E}[f(x^{s+1}_t) - f(x^*)].
\end{align*}
The second inequality uses unbiasedness of the \svrg update and convexity of $f$. The third inequality follows from Lemma~\ref{lem:var-lemma}. Defining the Lyapunov function
$$
P^s := \mathbb{E}[\|x^{s}_{m} - x^*\|^2] + 4mL\eta^2 \mathbb{E}[f(\tilde{x}^{s}) - f(x^*)],
$$
and summing the above inequality over $t$, we get
\begin{align*}
  2\eta(1 - 4L\eta) \sum_{t=0}^{m-1} \mathbb{E}[f(x^{s+1}_t) - f(x^*)] \leq P^{s} - P^{s+1}.
\end{align*}
This due is to the fact that 
\begin{align*}
P^{s+1} =  \mathbb{E}[\|x^{s+1}_{m} - x^*\|^2] + 4mL\eta^2 \mathbb{E}[f(\tilde{x}^{s+1}) - f(x^*)] \\
= \mathbb{E}[\|x^{s+1}_{m} - x^*\|^2] + 4L\eta^2 \sum_{t=0}^{m-1} \mathbb{E}[f(x^{s+1}_t) - f(x^*)].
\end{align*}
The above equality uses the fact that $p_m = 0$ and $p_i = 1/m$ for $0 \leq i < m$. Summing over all epochs and telescoping we then obtain
\begin{align*}
\mathbb{E}[f(x_a) - f(x^*)] \leq P^0\bigl(2T\eta(1 - 4L\eta)\bigr)^{-1}.
\end{align*}
The inequality also uses the definition of $x_a$ given in Alg~\ref{alg:svrg}. On this inequality we use Lemma~\ref{lem:grad-lemma}, which yields
\begin{align*}
&\mathbb{E}[\|\nabla f(x_a)\|^2] \leq 2L\mathbb{E}[f(x_a) - f(x^*)] \\ &\leq \frac{L\|x^{0} - x^*\|^2 + 4mL^2\eta^2 [f(x^{0}) - f(x^*)]}{T\eta(1 - 4L\eta)}.\qedhere
\end{align*}
\end{proof}

It is easy to see that we can obtain convergence rates for $E[f(x_a)- f(x^*)]$ from the above reasoning. This leads to a \emph{direct} analysis of $\svrg$ for convex functions. 


\section{Minibatch Nonconvex SVRG}

\section*{Proof of Theorem~\ref{thm:nonconvex-minibatch}}

\begin{algorithm}[tb]\small
   \caption{Mini-batch SVRG}
   \label{alg:minibatch-svrg}
\begin{algorithmic}[1]
   \STATE {\bfseries Input:} $\tilde{x}^0 = x^0_m = x^0 \in \mathbb{R}^d$,  epoch length $m$, step sizes $\{\eta_i > 0\}_{i=0}^{m-1}$, $S = \lceil T/m \rceil$, discrete probability distribution $\{p_i\}_{i=0}^{m}$, mini-batch size $b$
   \FOR{$s=0$ {\bfseries to} $S-1$}
   \STATE $x^{s+1}_0 = x^{s}_m$
   \STATE $g^{s+1} = \frac{1}{n} \sum_{i=1}^n \nabla f_{i}(\tilde{x}^{s})$
   \FOR{$t=0$ {\bfseries to} $m-1$}
   \STATE Choose a mini-batch (uniformly random with replacement) $I_t \subset [n]$ of size $b$
   \STATE $u_t^{s+1} =  \frac{1}{b} \sum_{i_t \in I_t} (\nabla f_{i_t}(x^{s+1}_t) - \nabla f_{i_t}(\tilde{x}^{s})) + g^{s+1}$
   \STATE $x^{s+1}_{t+1} = x^{s+1}_{t} - \eta_t u_t^{s+1} $
   \ENDFOR
   \STATE $\tilde{x}^{s+1} = \sum_{i=0}^{m} p_i x_{i}^{s+1}$
   \ENDFOR
   \STATE {\bfseries Output:} Iterate $x_a$ chosen uniformly random from $\{\{x^{s+1}_t\}_{t=0}^{m-1}\}_{s=0}^{S-1}$.
\end{algorithmic}
\end{algorithm}

The proofs essentially follow along the lines of Lemma~\ref{lem:nonconvex-svrg}, Theorem~\ref{thm:nonconvex-inter} and Theorem~\ref{thm:nonconvex-gen} with the added complexity of mini-batch. We first prove few intermediate results before proceeding to the proof of Theorem~\ref{thm:nonconvex-minibatch}.

\begin{lemma}
  \label{lem:nonconvex-minibatch-svrg}
  Suppose we have
    \begin{align*}
     &\overline{R}^{s+1}_{t} := \mathbb{E}[f(x^{s+1}_{t}) + \overline{c}_{t} \|x^{s+1}_{t} - \tilde{x}^{s}\|^2],  \\
     &\overline{c}_{t} = \overline{c}_{t+1}(1 + \eta_t\beta_t + \frac{2\eta_t^2L^2}{b} ) +  \frac{\eta_t^2L^3}{b},
     \end{align*}
     for $\ 0 \leq s \leq S-1$ and  $0 \leq t \leq m-1$ and the parameters $\eta_t, \beta_t$ and $\overline{c}_{t+1}$ are chosen such that
     $$
     \left(\eta_t - \frac{\overline{c}_{t+1}\eta_t}{\beta_t} - \eta_t^2L - 2\overline{c}_{t+1}\eta_t^2\right) \geq 0. 
     $$
     Then the iterates $x^{s+1}_t$ in the mini-batch version of Algorithm~\ref{alg:svrg} i.e., Algorithm~\ref{alg:minibatch-svrg} with mini-batch size $b$ satisfy the bound:
  \begin{align*}
    \mathbb{E}[\|\nabla f(x^{s+1}_{t})\|^2] \leq \frac{\overline{R}^{s+1}_{t} - \overline{R}^{s+1}_{t+1}}{\left(\eta_t - \frac{\overline{c}_{t+1}\eta_t}{\beta_t} - \eta_t^2L - 2\overline{c}_{t+1}\eta_t^2\right)},
  \end{align*}
\end{lemma}
\begin{proof}
Using essentially the same argument as the proof of Lemma.~\ref{lem:nonconvex-svrg} until Equation~\eqref{eq:svrg-proof-eq3}, we have
\begin{align}
& \overline{R}^{s+1}_{t+1} \leq \mathbb{E}[f(x^{s+1}_{t}) - \left(\eta_t - \tfrac{\overline{c}_{t+1}\eta_t}{\beta_t}\right) \|\nabla f(x^{s+1}_{t})\|^2 \nonumber\\
&  \qquad \qquad + \left(\tfrac{L\eta_t^2}{2} + \overline{c}_{t+1}\eta_t^2 \right)\mathbb{E}[\|u^{s+1}_t\|^2] \nonumber\\
&  \qquad \qquad + \left( \overline{c}_{t+1} + \overline{c}_{t+1}\eta_t\beta_t \right) \mathbb{E}\left[\| x^{s+1}_t - \tilde{x}^s \|^2 \right].
\label{eq:minibatch-svrg-proof-eq3}
\end{align}
We use Lemma~\ref{lem:nonconvex-minibatch-variance-lemma} in order to bound $\mathbb{E}[\|u^{s+1}_{t}\|^2]$ in the above inequality. Substituting it in Equation~\eqref{eq:minibatch-svrg-proof-eq3}, we see that
\begin{align*}
& \overline{R}^{s+1}_{t+1} \leq \mathbb{E}[f(x^{s+1}_{t})] \nonumber \\
& \qquad - \left(\eta_t - \tfrac{\overline{c}_{t+1}\eta_t}{\beta_t} - \eta_t^2L - 2\overline{c}_{t+1}\eta_t^2\right) \mathbb{E}[\|\nabla f(x^{s+1}_{t})\|^2] \nonumber\\
& \qquad + \left[\overline{c}_{t+1}\bigl(1 + \eta_t\beta_t + \tfrac{2\eta_t^2L^2}{b}\bigr)+\tfrac{\eta_t^2L^3}{b}\right] 
  \mathbb{E}\left[\| x^{s+1}_t - \tilde{x}^s \|^2 \right] \nonumber \\
& \leq \overline{R}^{s+1}_{t} - \bigl(\eta_t - \tfrac{\overline{c}_{t+1}\eta_t}{\beta_t} - \eta_t^2L - 2\overline{c}_{t+1}\eta_t^2\bigr) \mathbb{E}[\|\nabla f(x^{s+1}_{t})\|^2].
\end{align*}
The second inequality follows from the definition of $\overline{c}_{t}$ and $\overline{R}^{s+1}_{t}$, thus concluding the proof. 
\end{proof}

Our intermediate key result is the following theorem that provides convergence rate of mini-batch $\svrg$.
\begin{theorem}
  Let $\overline{\gamma}_n$ denote the following quantity:
  \begin{align*}
    \overline{\gamma}_n := \min_{0 \leq t \leq m-1}\quad\bigl(\eta - \tfrac{\overline{c}_{t+1}\eta}{\beta} - \eta^2L - 2\overline{c}_{t+1}\eta^2\bigr).
  \end{align*}
  Suppose $\eta_t = \eta$ and $\beta_t = \beta$ for all $t \in \{0, \dots, m-1\}$, $\overline{c}_m = 0$, $\overline{c}_{t} = \overline{c}_{t+1}(1 + \eta_t\beta_t + \frac{2\eta_t^2L^2}{b} ) +  \frac{\eta_t^2L^3}{b}$ for $t \in \{0, \dots, m-1\}$  and $\overline{\gamma}_n > 0$. Further, let $p_{m} = 1$ and $p_{i} = 0$ for $0 \leq i < m$. Then for the output $x_a$ of mini-batch version of Algorithm~\ref{alg:svrg} with mini-batch size $b$, we have
  \begin{align*}
    \mathbb{E}[\|\nabla f(x_a)\|^2] \leq \frac{f(x^{0}) - f(x^*)}{T\overline{\gamma}_n},
  \end{align*}
  where $x^*$ is an optimal solution to~\eqref{eq:1}.
  \label{thm:minibatch-nonconvex-inter}
\end{theorem}
\begin{proof}
  Since $\eta_t = \eta$ for $t \in \{0, \dots, m-1\}$, using  Lemma~\ref{lem:nonconvex-minibatch-svrg} and telescoping the sum, we obtain
\begin{align*}
  \nlsum_{t=0}^{m-1} \mathbb{E}[\|\nabla f(x^{s+1}_{t})\|^2] \leq \frac{\overline{R}^{s+1}_{0} - \overline{R}^{s+1}_{m}}{\overline{\gamma}_n}.
\end{align*}
This inequality in turn implies that
\begin{align*}
  \nlsum_{t=0}^{m-1} \mathbb{E}[\|\nabla f(x^{s+1}_{t})\|^2] \leq \frac{\mathbb{E}[f(\tilde{x}^s) - f(\tilde{x}^{s+1})]}{\overline{\gamma}_n},
\end{align*}
where we used that $\overline{R}^{s+1}_{m} = \mathbb{E}[f(x^{s+1}_m)] = \mathbb{E}[f(\tilde{x}^{s+1})]$ (since $\overline{c}_m = 0$, $p_{m} = 1$, and $p_i = 0$ for $i < m$), and that $\overline{R}^{s+1}_{0} = \mathbb{E}[f(\tilde{x}^s)]$ (since $x^{s+1}_0 = \tilde{x}^s$, as $p_{m} = 1$ and $p_i = 0$ for $i < m$). Now sum over all epochs and using the fact that $\tilde{x}^0 = x^0$, we get the desired result.
\end{proof}

We now present the proof of Theorem~\ref{thm:nonconvex-minibatch} using the above results.
\begin{theorem*}
 Let $\overline{\gamma}_n$ denote the following quantity:
  \begin{align*}
    \overline{\gamma}_n := \min_{0 \leq t \leq m-1}\quad\bigl(\eta - \tfrac{\overline{c}_{t+1}\eta}{\beta} - \eta^2L - 2\overline{c}_{t+1}\eta^2\bigr).
  \end{align*}
  where $\overline{c}_m = 0$, $\overline{c}_{t} = \overline{c}_{t+1}(1 + \eta\beta + \nicefrac{2\eta^2L^2}{b} ) +  \nicefrac{\eta_t^2L^3}{b}$ for \fromto{0}{t}{m-1}. Suppose $\eta = \mu_2b/(Ln^{2/3})$ ($0 < \mu_2 < 1$), $\beta = L/n^{1/3}$, $m = \lfloor n/(3b\mu_2) \rfloor$ and $T$ is some multiple of $m$. Then for the mini-batch version of Algorithm~\ref{alg:svrg} with mini-batch size $b < n^{2/3}$, there exists universal constants $\mu_2, \nu_2 > 0$ such that we have the following: $\overline{\gamma}_n \geq \frac{\nu_2b}{Ln^{2/3}}$ and
\begin{align*}
\mathbb{E}[\|\nabla f(x_a)\|^2] &\leq \frac{Ln^{2/3} [f(x^{0}) - f(x^*)]}{bT\nu_2},
\end{align*} 
where $x^*$ is optimal for~\eqref{eq:1}.
\end{theorem*}
\begin{proof}[Proof of Theorem~\ref{thm:nonconvex-minibatch}]
  We first observe that using the specified values of $\beta$ and $\eta$ we obtain
  \begin{align*}
    \overline{\theta} := \frac{2\eta^2L^2}{b} + \eta\beta = \frac{2\mu_2^2b}{n^{4/3}} + \frac{\mu_2 b}{n} \leq \frac{3\mu_2 b}{n}.
  \end{align*}
  The above inequality follows since $\mu_2 \leq 1$ and $n \geq 1$. For our analysis, we will require the following bound on $\overline{c}_{0}$:
  \begin{align}
    \label{eq:minibatch-c0-bound}
    \overline{c}_0 &= \frac{\mu_2^2b^2L}{bn^{4/3}} \frac{(1 + \overline{\theta})^m - 1}{\overline{\theta}} = \frac{\mu_2bL ((1 + \overline{\theta})^m - 1)}{2b\mu_2 + bn^{1/3}} \nonumber \\
        &\leq n^{-1/3}(\mu_2L (e - 1)),
  \end{align}
  wherein the first equality holds due to the relation $\overline{c}_{t} = \overline{c}_{t+1}(1 + \eta_t\beta_t + \tfrac{2\eta_t^2L^2}{b}) +  \tfrac{\eta_t^2L^3}{b}$, and the inequality follows upon again noting that $(1 + 1/l)^l$ is increasing for $l>0$ and  $\lim_{l \rightarrow \infty} (1 + \frac{1}{l})^l = e$. Now we can lower bound $\overline{\gamma}_n$, as
  \begin{align*}
    \overline{\gamma}_n &= \min_t \bigl(\eta - \tfrac{\overline{c}_{t+1}\eta}{\beta} - \eta^2L - 2\overline{c}_{t+1}\eta^2\bigr) \\
             &\geq \bigl(\eta - \tfrac{\overline{c}_{0}\eta}{\beta} - \eta^2L - 2\overline{c}_{0}\eta^2\bigr) \geq \frac{b\nu_2}{Ln^{2/3}},
  \end{align*}
  where $\nu_2$ is a constant independent of $n$. The first inequality holds since $\overline{c}_t$ decreases with $t$. The second one holds since (a) $\overline{c}_0/\beta$ is upper bounded by a constant independent of $n$ as $\overline{c}_0/\beta \leq \mu_2(e-1)$ (due to Equation~\eqref{eq:minibatch-c0-bound}), (b) $\eta^2L \leq \mu_2\eta$ (as $b < n^{2/3}$) and (c) $2\overline{c}_{0}\eta^2 \leq 2\mu_2^2(e-1)\eta$ (again due to Equation~\eqref{eq:minibatch-c0-bound} and the fact $b < n^{2/3}$). By choosing an appropriately small constant $\mu_2$ (independent of n), one can ensure that $\overline{\gamma}_n \geq {b\nu_2}/(Ln^{2/3})$ for some universal constant $\nu_2$. For example, choosing $\mu_2 = 1/4$, we have $\overline{\gamma}_n \geq {b\nu_2}/(Ln^{2/3})$ with $\nu_2 = 1/40$. Substituting the above lower bound in Theorem~\ref{thm:minibatch-nonconvex-inter}, we get the desired result.
\end{proof}

\section{MSVRG: Convergence Rate}

\section*{Proof of Theorem~\ref{thm:nonconvex-svrg-best}}
\begin{theorem*}
Let $f \in \Fc_n$ have $\sigma$-bounded gradients. Let $\eta_t = \eta = \max\{\nicefrac{c}{\sqrt{T}},\nicefrac{\mu_1}{(Ln^{2/3})}\}$ ($\mu_1$ is the universal constant from Corollary~\ref{cor:nonconvex}), $m = \lfloor\nicefrac{n}{(3\mu_1)}\rfloor$, and $c = \sqrt{\frac{f(x^0) - f(x^*)}{2L\sigma^2}}$. Further, let 
$T$ be a multiple of $m$, $p_{m} = 1$, and $p_{i} = 0$ for $0{\;\leq\;}i{\;<\;}m$. Then, the output $x_a$ of Algorithm~\ref{alg:svrg} satisfies
\begin{align*}
& \mathbb{E}[\|\nabla f(x_a)\|^2] \\
&\leq \bar{\nu} \min\Big\lbrace 2\sqrt{\frac{2(f(x^0) - f(x^{*}))L}{T}} \sigma, \frac{ L n^{2/3} [f(x^{0}) - f(x^*)]}{T\nu_1}\Big\rbrace,
\end{align*}
where $\bar{\nu}$ is a universal constant, $\nu_1$ is the universal constant from Corollary~\ref{cor:nonconvex} and $x^*$ is an optimal solution to~\eqref{eq:1}.
\end{theorem*}
\begin{proof}
First, we observe that the step size $\eta$ is chosen to be $\max\{c/\sqrt{T},\mu_1/(Ln^{2/3})\}$ where
$$
c = \sqrt{\frac{f(x^0) - f(x^*)}{2L\sigma^2}}.
$$
Suppose $\eta = \mu_1/(Ln^{2/3})$, we obtain the convergence rate in Corollary~\ref{cor:nonconvex}. Now, lets consider the case where $\eta = c/\sqrt{T}$. In this case, we have the following bound:
\begin{align*}
&\mathbb{E}[\|v^{s+1}_t\|^2] \\
&= \mathbb{E}[\|\nabla f_{i_t}(x^{s+1}_{t}) - \nabla f_{i_t}(\tilde{x}^{s}) + \nabla f(\tilde{x}^{s}) \|^2] \\
&\leq 2\left(\mathbb{E}[\|\nabla f_{i_t}(x^{s+1}_{t})\|^2 +  \|\nabla f_{i_t}(\tilde{x}^{s}) - \nabla f(\tilde{x}^{s})\|^2]\right) \\
&\leq 2\left(\mathbb{E}[\|\nabla f_{i_t}(x^{s+1}_{t})\|^2 +  \|\nabla f_{i_t}(\tilde{x}^{s})\|^2]\right) \\
&\leq 4\sigma^2.
\end{align*}
The first inequality follows from Lemma~\ref{lem:r-lemma} with $r=2$. The second inequality follows from (a) $\sigma$-bounded gradient property of $f$ and (b) the fact that for a random variable $\zeta$, $\mathbb{E}[\|\zeta - \mathbb{E}[\zeta]\|^2] \leq \mathbb{E}[\|\zeta\|^2]$.
The rest of the proof is along exactly the lines as in Theorem~\ref{thm:sgd-conv}. This provides a convergence rate similar to Theorem~\ref{thm:sgd-conv}. More specifically, using step size $c/\sqrt{T}$, we get
\begin{align}
\label{eq:svrg-best-eq1}
\mathbb{E}[\|f(x_a)\|^2] \leq 2\sqrt{\frac{2(f(x^0) - f(x^{*}))L}{T}} \sigma.
\end{align}
The only thing that remains to be proved is that with the step size choice of $\max\{c/\sqrt{T},\mu_1/(Ln^{2/3})\}$, the minimum of two bounds hold. Consider the case $c/\sqrt{T} > \mu_1/(Ln^{2/3})$. In this case, we have the following:
\begin{align*}
\frac{2\sqrt{\frac{2(f(x^0) - f(x^{*}))L}{T}} \sigma}{\frac{Ln^{2/3} [f(x^{0}) - f(x^*)]}{T\nu_1}} &= \frac{2\nu_1\sigma\sqrt{2LT}}{Ln^{2/3}\sqrt{f(x^0) - f(x^{*})}} \\
&\leq 2\nu_1/\mu_1 \leq \bar{\nu} := \max\left\lbrace \frac{2\nu_1}{\mu_1}, \frac{\mu_1}{2\nu_1} \right\rbrace,
\end{align*}
where $\nu_1$ is the constant in Corollary~\ref{cor:nonconvex}. This inequality holds since $c/\sqrt{T} > \mu_1/(Ln^{2/3})$. Rearranging the above inequality, we have
$$
2\sqrt{\frac{2(f(x^0) - f(x^{*}))L}{T}} \sigma \leq \frac{\bar{\nu}Ln^{2/3} [f(x^{0}) - f(x^*)]}{T}
$$
in this case. Note that the left hand side of the above inequality is precisely the bound obtained by using step size $c/\sqrt{T}$ (see Equation~\eqref{eq:svrg-best-eq1}). Similarly, when $c/\sqrt{T} \leq \mu_1/(Ln^{2/3})$, the inequality holds in the other direction. Using these two observations, we have the desired result.
\end{proof}

\section{Key Lemmatta}
\label{sec:key-lemmas}
\begin{lemma}
\label{lem:nonconvex-variance-lemma}
  For the intermediate iterates $v^{s+1}_t$ computed by Algorithm~\ref{alg:svrg}, we have the following:
  \begin{align*}
    \mathbb{E}[\|v^{s+1}_t\|^2] \leq 2\mathbb{E}[\|\nabla f(x^{s+1}_{t})\|^2] + 2L^2 \mathbb{E}[\|x^{s+1}_{t} - \tilde{x}^{s}\|^2].
\end{align*}
\end{lemma}
\begin{proof}
The proof simply follows from the proof of Lemma~\ref{lem:nonconvex-minibatch-variance-lemma} with $I_t = \{i_t\}$.
\end{proof}

We now present a result to bound the variance of mini-batch $\svrg$. 
\begin{lemma}
\label{lem:nonconvex-minibatch-variance-lemma}
  Let $u^{s+1}_t$ be computed by the mini-batch version of Algorithm~\ref{alg:svrg} i.e., Algorithm~\ref{alg:minibatch-svrg} with mini-batch size $b$. Then,
  \begin{align*}
    \mathbb{E}[\|u^{s+1}_t\|^2] \leq 2\mathbb{E}[\|\nabla f(x^{s+1}_{t})\|^2] + \tfrac{2L^2}{b} \mathbb{E}[\|x^{s+1}_{t} - \tilde{x}^{s}\|^2].
\end{align*}
\end{lemma}
\begin{proof}
For the ease of exposition, we use the following notation:
$$
\zeta_{t}^{s+1} = \frac{1}{|I_t|} \sum_{i_t \in I_t} \left(\nabla f_{i_t}(x^{s+1}_{t}) - \nabla f_{i_t}(\tilde{x}^{s})\right).
$$
We use the definition of $u^{s+1}_t$ to get
\begin{align*}
&\mathbb{E}[\|u^{s+1}_t\|^2] = \mathbb{E}[\|\zeta_{t}^{s+1} + \nabla f(\tilde{x}^{s}) \|^2] \\
&= \mathbb{E}[\| \zeta_{t}^{s+1} + \nabla f(\tilde{x}^{s}) - \nabla f(x^{s+1}_{t}) +  \nabla f(x^{s+1}_{t}) \|^2]\\
&\leq 2\mathbb{E}[\|\nabla f(x^{s+1}_{t})\|^2] + 2 \mathbb{E}[\|\zeta_{t}^{s+1} - \mathbb{E}[\zeta_{t}^{s+1}]\|^2] \\
&= 2\mathbb{E}[\|\nabla f(x^{s+1}_{t})\|^2] + \frac{2}{b^2} \mathbb{E}\left[\left\|\sum_{i_t \in I_t} \left(\nabla f_{i_t}(x^{s+1}_{t}) - \nabla f_{i_t}(\tilde{x}^{s}) - \mathbb{E}[\zeta_{t}^{s+1}] \right)  \right\|^2 \right]
\end{align*}
The first inequality follows from Lemma~\ref{lem:r-lemma} (with $r = 2$) and the fact that $ \mathbb{E}[\zeta_{t}^{s+1}] = \nabla f(x^{s+1}_{t}) - \nabla f(\tilde{x}^{s})$. From the above inequality, we get
\begin{align*}
&\mathbb{E}[\|u^{s+1}_t\|^2] \\
& \leq 2\mathbb{E}[\|\nabla f(x^{s+1}_{t})\|^2] + \frac{2}{b} \mathbb{E}[\|\nabla f_{i_t}(x^{s+1}_{t}) - \nabla f_{i_t}(\tilde{x}^{s})\|^2] \\
& \leq 2\mathbb{E}[\|\nabla f(x^{s+1}_{t})\|^2] + \frac{2L^2}{b} \mathbb{E}[\|x^{s+1}_{t} - \tilde{x}^{s}\|^2]
\end{align*}
The first inequality follows from the fact that the indices $i_t$ are drawn uniformly randomly and independently from $\{1, \dots, n\}$  and noting that for a random variable $\zeta$, $\mathbb{E}[\|\zeta - \mathbb{E}[\zeta]\|^2] \leq \mathbb{E}[\|\zeta\|^2]$. The last inequality follows from $L$-smoothness of $f_{i_t}$.
\end{proof}

\section{Experiments}
\label{sec:remaining-expts}

\begin{figure*}[!t]
\centering
   \begin{minipage}[b]{.3\textwidth}
   \includegraphics[width=\textwidth]{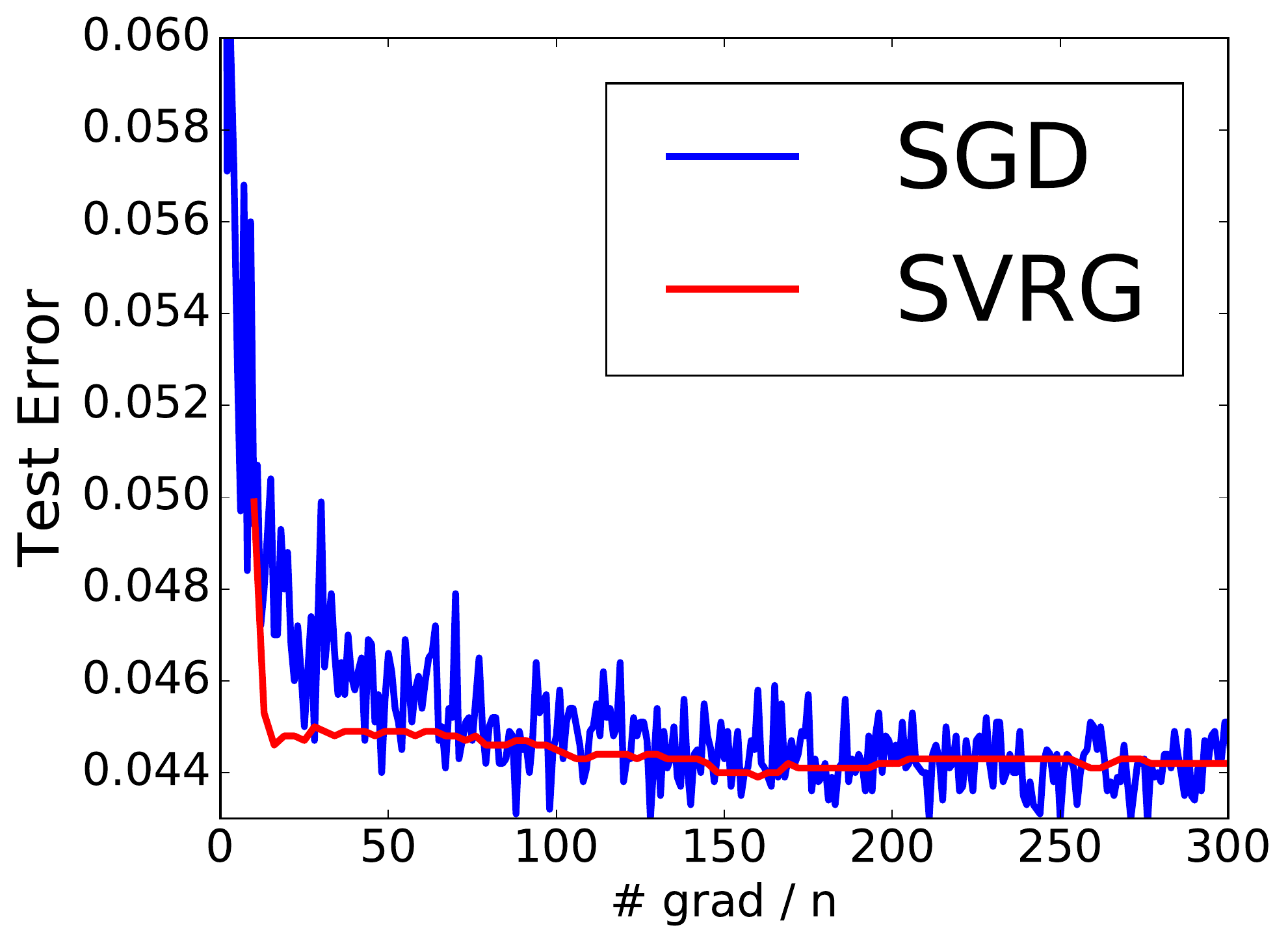}
   \end{minipage} %
   \begin{minipage}[b]{.3\textwidth}
   \includegraphics[width=\textwidth]{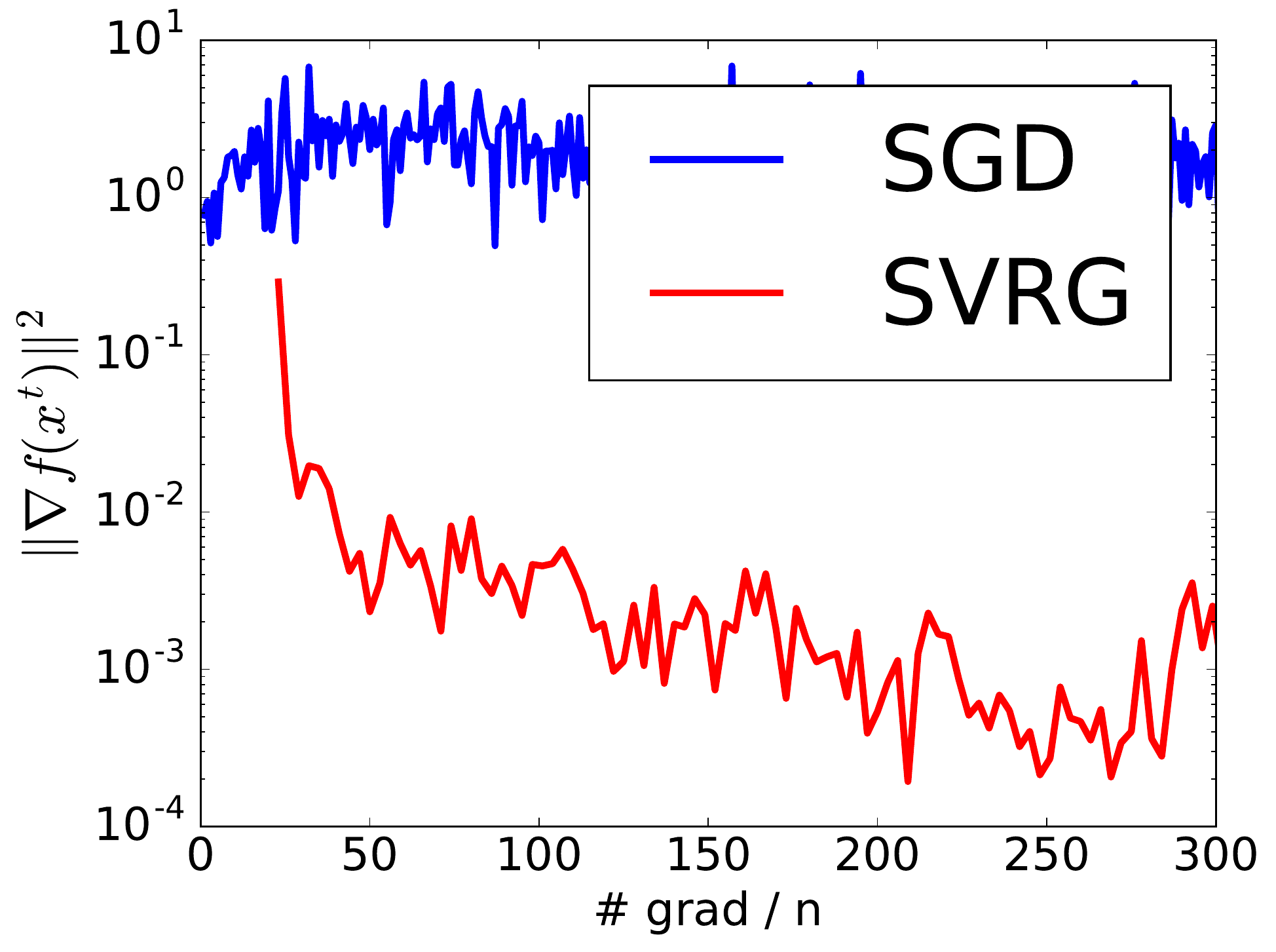}
   \end{minipage}
   \begin{minipage}[b]{.3\textwidth}
   \includegraphics[width=\textwidth]{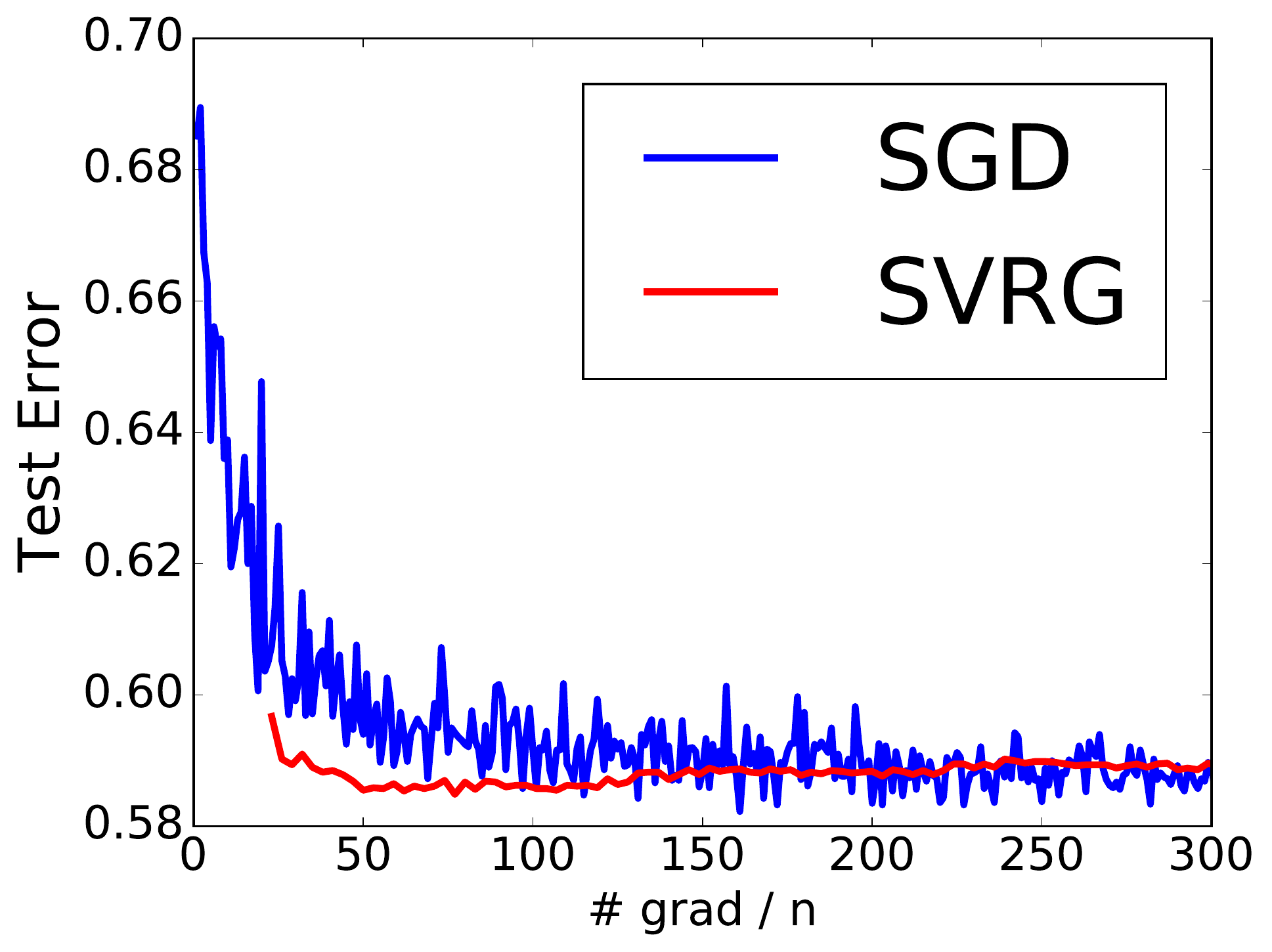}
   \end{minipage}
	\caption{Neural network results for MNIST and STL-10. The leftmost result is for MNIST. The remaining two plots are of STL-10.}
	\label{fig:results2}
\end{figure*}

Figure~\ref{fig:results2} shows the remaining plots for MNIST and STL-10 datasets. As seen in the plots, there is no significant difference in the test error of $\svrg$ and $\sgd$ for these datasets.

\section{Other Lemmas}
We need Lemma~\ref{lem:grad-lemma} for our results in the convex case.

\begin{lemma}[\protect{\citet{Johnson13}}]
\label{lem:grad-lemma}
Let $g:\mathbb{R}^d \rightarrow \mathbb{R}$ be convex with $L$-Lipschitz continuous gradient. Then,
\begin{align*}
  \|\nabla g(x) - \nabla g(y)\|^2 \leq 2L[g(x) - g(y) - \langle \nabla g(y), x - y \rangle],
\end{align*}
for all $x, y \in \mathbb{R}^d$.
\end{lemma}
\begin{proof}
Consider $h(x) := g(x) - g(y) - \langle \nabla g(y), x - y \rangle$ for arbitrary $y \in \mathbb{R}^d$. Observe that $\nabla h$ is also $L$-Lipschitz continuous. Note that $h(x) \geq 0$ (since  $h(y) = 0$ and  $\nabla h(y) = 0$, or alternatively since $h$ defines a Bregman divergence), from which it follows that
\begin{align*}
  0 &\leq \min_{\rho} [h(x - \rho \nabla h(x))] \\
    &\leq \min_{\rho} [h(x) - \rho \|\nabla h(x)\|^2 + \tfrac{L\rho^2}{2}\|\nabla h(x)\|^2] \\
    &= h(x) - \tfrac{1}{2L}\|\nabla h(x)\|^2.
\end{align*}
Rewriting in terms of $g$ we obtain the required result.
\end{proof}

Lemma~\ref{lem:var-lemma} bounds the variance of $\svrg$ for the convex case. Please refer to \citep{Johnson13} for more details.

\begin{lemma}[\protect{\citep{Johnson13}}]
\label{lem:var-lemma}
Suppose $f_i$ is convex for all $i \in [n]$. For the updates in Algorithm~\ref{alg:svrg} we have the following inequality:
\begin{align*}
\mathbb{E}[\|v^{s+1}_t\|^2] \leq 4L [f(x^{s+1}_t) - f(x^*) + f(\tilde{x}^{s} - f(x^*)].
\end{align*}
\end{lemma}
\begin{proof}
The proof follows upon observing the following:
\begin{align*}
&\mathbb{E}[\|v^{s+1}_t\|^2 = \mathbb{E}[\|\nabla f_{i_t}(x^{s+1}_{t}) - \nabla f_{i_t}(x^{s+1}_{0}) + \nabla f(\tilde{x}^{s}) \|^2] \\
&\leq 2 \mathbb{E}[\|\nabla f_{i_t}(x^{s+1}_t) - \nabla f_{i_t}(x^*)\|^2] \\
& \ \ \ + 2 \mathbb{E}[\|\nabla f_{i_t}(\tilde{x}^{s}) - \nabla f_{i_t}(x^*) - (\nabla f(\tilde{x}^{s}) - \nabla f(x^*))\|^2] \\
&\leq 2 \mathbb{E}[\|\nabla f_{i_t}(x^{s+1}_t) - \nabla f_{i_t}(x^*)\|^2] \\
& \qquad \qquad + 2 \mathbb{E}[\|\nabla f_{i_t}(\tilde{x}^{s}) - \nabla f_{i_t}(x^*)\|^2] \\
& \leq 4L [f(x^{s+1}_t - f(x^*) + f(\tilde{x}^{s}) - f(x^*)].
\end{align*}
The first inequality follows from Cauchy-Schwarz and Young inequality; the second one from  $\mathbb{E}[\|\xi - \mathbb{E}[\xi]\|^2] \leq \mathbb{E}[\|\xi\|^2]$, and the third one from Lemma~\ref{lem:grad-lemma}.
\end{proof}

\begin{lemma}
\label{lem:r-lemma}
For random variables $z_1, \dots, z_r$, we have
\begin{align*}
\mathbb{E}\left[ \|z_1 + ... + z_r\|^2 \right] \leq r \mathbb{E}\left[\|z_1\|^2 + ... + \|z_r\|^2\right].
\end{align*}
\end{lemma}